\documentclass[11pt]{amsart}
\usepackage[latin1]{inputenc}
\usepackage{amsmath,amsthm}
\usepackage{amssymb,amsfonts}
\usepackage{hyperref}
\usepackage{tikz}

\usepackage[shortlabels]{enumitem}
\usepackage{bm}
\usepackage{cite}
\usepackage{graphicx}
\usepackage{verbatim}
\usepackage{setspace}
\usepackage{color}

\usepackage[a4paper]{geometry}

\hypersetup{bookmarks,final,colorlinks=true,
linkcolor=[rgb]{0.1, 0.1, 0.44},
citecolor=[rgb]{0.1, 0.1, 0.44}}

\usepackage{stmaryrd}
\DeclareSymbolFont{bbold}{U}{bbold}{m}{n}
\DeclareSymbolFontAlphabet{\mathbbm}{bbold}

\title{Tameness for set theory $I$}
\author{Matteo Viale}
\thanks{The author acknowledge support from INDAM through GNSAGA and from the project:
\emph{PRIN 2017-2017NWTM8R
Mathematical Logic: models, sets, computability.
}
}

\theoremstyle{plain}
	\newtheorem{Theorem}{Theorem}
	\newtheorem{Fact}{Fact}
	\newtheorem{Notation}{Notation}
	\newtheorem{Remark}{Remark}
	
	\newtheorem{theorem}{Theorem}[section]
	\newtheorem{proposition}[theorem]{Proposition}
	\newtheorem{lemma}[theorem]{Lemma}
	\newtheorem{corollary}[theorem]{Corollary}
	\newtheorem{fact}[theorem]{Fact}

	\newtheorem{claim}{Claim}
	
\theoremstyle{definition}
	\newtheorem{Definition}{Definition}

	\newtheorem{definition}[theorem]{Definition}
	\newtheorem{notation}[theorem]{Notation}

\theoremstyle{remark}
	\newtheorem{remark}[theorem]{Remark}

\newcommand{\ZFC}{\ensuremath{\mathsf{ZFC}}}
\newcommand{\ZF}{\ensuremath{\mathsf{ZF}}}

\newcommand{\WFE}{\ensuremath{\mathsf{WFE}}}

\DeclareMathOperator{\dom}{dom}
\DeclareMathOperator{\ran}{ran}

\DeclareMathOperator{\otp}{otp}

\DeclareMathOperator{\Coll}{Coll}

\newcommand{\maxUB}{\ensuremath{
{\mathbf{MAX}(\mathsf{UB})}}}
\newcommand{\Pmax}{\ensuremath{\mathbb{P}_{\mathrm{max}}}}
\newcommand{\NS}{\ensuremath{\mathbf{NS}}} 
\newcommand{\stUB}{\ensuremath{(*)\text{-}\mathsf{UB}}}
\newcommand{\stA}{\ensuremath{(*)\text{-}\mathcal{A}}}

\newcommand{\bool}[1]{\mathsf{#1}}
\newcommand{\tow}[1]{\mathcal{#1}}

\newcommand{\pow}[1]{\mathcal{P}\left(#1\right)}

\newcommand{\ap}[1]{\langle #1 \rangle}
\newcommand{\bp}[1]{\left\lbrace #1 \right\rbrace}

\newcommand{\Cod}{\ensuremath{\text{{\rm Cod}}}}
\newcommand{\ST}{\ensuremath{\text{{\sf ST}}}}
\newcommand{\UB}{\ensuremath{\text{{\sf UB}}}}
\newcommand{\lUB}{\ensuremath{\text{{\rm l-UB}}}}

\newcommand{\MM}{\ensuremath{\text{{\sf MM}}}}

\begin{document}

\begin{abstract}
The paper is a first of two and aims to show that (assuming large cardinals) set theory is 
a tractable (and we dare to say tame) 
first order theory when formalized in a first order signature with natural predicate symbols for the basic 
definable concepts of second and third order arithmetic, and appealing
to the model-theoretic notions of model completeness and model companionship. 

Specifically we develop a general framework linking 
generic absoluteness results to model companionship and
show that (with the required care in details) 
a $\Pi_2$-property formalized in an appropriate language for second or third order number theory 
is forcible from some 
$T\supseteq\ZFC+$\emph{large cardinals}
if and only if it is consistent with the universal fragment of $T$
if and only if it is realized in the model companion of $T$.

The paper is accessible to any person who has a fair acquaintance with set theory
and first order logic at the level of an under-graduate course in both topics; however bizarre this may appear (given the results we aim to prove) 
no knowledge of forcing or large cardinals is required to get the proofs of its main results
(if one accepts as black-boxes the relevant generic absoluteness results).
On the other hand familiarity with the notions of model completeness and model companionship is essential.
All the necessary model-theoretic background will be given in full detail.

The present work
expands and systematize
previous results obtained with Venturi. 
 \end{abstract}

\maketitle


The key model-theoretic result of this paper is 
that the definable\footnote{According to \cite[Ch. I.13]{KUNEN}.} (and conservative) extension of any $T\supseteq\ZFC$ introducing predicates for the $\Delta_0$-definable (class) relations, function symbols for the $\Delta_0$-definable (class) functions, and predicates for the
lightface definable projective
subsets of $\pow{\kappa}$ has as \emph{model companion} the $T$-provable 
fragment of the theory of 
$H_{\kappa^+}$ in this signature (cfr. Thm. \ref{Thm:mainthm-1}).

We also give evidence that any existence proof of the model companion of some $T$ extending 
$\ZFC+$\emph{large cardinals} comes in pairs with generic absoluteness results for $T$.

Specifically we use Thm. \ref{Thm:mainthm-1} (and variations of it) 
to show that these results couple perfectly with Woodin's generic absoluteness
for second order number theory (cfr Thm. \ref{Thm:mainthm-3}),  the theory of $H_{\aleph_2}$ assuming 
Woodin's axiom $(*)$ (cfr Thm. \ref{Thm:mainthm-4}, Thm. \ref{Thm:mainthm-7}, Thm. \ref{Thm:mainthm-6}), and the author's generic absoluteness results for the theory of $H_{\aleph_2}$ (cfr Thm. \ref{Thm:mainthm-5}).

We proceed stating our main results.

\begin{Notation}\label{not:modthnot}
Let $T$ be a $\tau$-theory.
$T_\forall$ is the family of $\Pi_1$-sentences\footnote{A $\Pi_n$-formula for $\tau$ relative to $T$
is a $\tau$-formula admitting a $T$-equivalent formula
in prenex normal form with blocks of alternating quantifiers with $\forall$ being its leftmost 
quantifier. Accordingly we define $\Sigma_n$-formulae relative to $T$.
The $\Delta_n$-formulae relative to $T$ are those which are both $\Pi_n$ and $\Sigma_n$. 
We also say universal $\tau$-formula for a $\Pi_1$-formula for $\tau$ and existential $\tau$-formula 
for a $\Sigma_1$-formula for $\tau$.} 
$\psi$ for $\tau$ which are provable 
from $T$.
Accordingly we define $T_\exists$, $T_{\forall\exists}$, etc.
\end{Notation}

Let $\tau_\ST$ be a signature containing predicate symbols 
$R_\psi$ of arity $m$ for all bounded\footnote{A formula is bounded if all its quantifiers are bounded to range over its free variables or constant symbols (see \cite[Def. IV.3.5]{KUNEN}).} 
$\in$-formulae $\psi(x_1,\dots,x_m)$, function symbols
$f_\theta$ of arity $k$ for for all bounded
$\in$-formulae $\theta(y,x_1,\dots,x_k)$, constant symbols 
$\omega$ and $\emptyset$.
$\ZFC_\ST\supseteq\ZFC$ is the $\tau_\ST$-theory
obtained adding axioms which force in each of its $\tau_\ST$-models $\emptyset$ to be interpreted by the empty set,
$\omega$ to be interpreted by the first infinite ordinal,
each $R_\psi$ as the class of $k$-tuples defined by the bounded formula
$\psi(x_1,\dots,x_k)$, each $f_\theta$ as the $l$-ary class function whose graph
is the extension of the bounded formula $\theta(x_1,\dots,x_l,y)$
(whenever $\theta$ defines a functional relation).
Essentially $\ZFC_\ST$ is set theory axiomatized in
a language admitting predicate symbols for $\Delta_0$-predicates, 
$\Delta_0$-definable functions, and a constant for the first infinite cardinal
(see Notation \ref{not:keynotation} and Fact \ref{fac:basicfact} below for details).

Let $\sigma_\ST$ be a signature containing predicate symbols $S_\psi$ of arity $n$ for all 
$\tau_\ST$-formulae $\phi(x_1,\dots,x_n)$; let 
Let $\sigma_\omega=\sigma_\ST\cup\tau_\ST$. 
$\ZFC^*_\omega\supseteq\ZFC_\ST$ is the $\sigma_\omega$-theory
obtained adding axioms which force in each of its $\sigma_\omega$-models
each predicate symbol $S_\phi$ of arity $n$ to be interpreted as the subset of 
$\pow{\omega^{<\omega}}^n$ defined by\footnote{For a set or definable class 
$Z$ and a $\tau_\ST$-formula $\psi$, 
$\psi^Z$ denotes the $\tau_\ST$-formula obtained from $\psi$ requiring all its quantifiers 
to range over $Z$.}  the $\tau_\ST$-formula
$\phi^{\pow{\omega^{<\omega}}}(x_1,\dots,x_n)$.
Essentially $\ZFC^*_\omega$ extends $\ZFC_\ST$
adding predicate symbols for the lightface definable projective sets\footnote{We decide to use $\pow{\omega^{<\omega}}$ rather than $\pow{\omega}$
(or any other uncountable Polish space)
to simplify slightly the coding devices we are going to implement to prove Thm. \ref{Thm:mainthm-3}. 
Similar considerations brings us to focus on $\pow{\omega_1^{<\omega}}$ rather than 
$\pow{\omega_1}$ in the formulation of Thm. \ref{Thm:mainthm-5} and on 
$\pow{\kappa^{<\omega}}$ rather than $\pow{\kappa}$ in the formulation of Thm. \ref{Thm:mainthm-1}. At the prize of complicating slightly the relevant proofs one can choose to replace $\alpha^{<\omega}$ by $\alpha$ all over for $\alpha$ any among $\omega,\omega_1,\kappa$.} (again see Notation \ref{not:keynotation} and Fact \ref{fac:basicfact} below for details).

\begin{Theorem}\label{Thm:mainthm-3}
Let $T$ be a $\sigma_\omega$-theory such that\footnote{It is not relevant for this paper to define Woodin cardinals. A definition is given in \cite[Def. 1.5.1]{STATLARSON}, for example.}
\[
T\supseteq \ZFC^*_\omega+\text{ there are class many Woodin cardinals}.
\]
Then $T$ has a model companion $T^*$.
Moreover TFAE for any
$\Pi_2$-sentence $\psi$ for $\sigma_\omega$:
\begin{enumerate}
\item \label{Thm:mainthm-3-2}
For all universal $\sigma_{\omega}$-sentences $\theta$ such that
$T+\theta$ is consistent,  so is $T_\forall+\theta+\psi$;
\item \label{Thm:mainthm-3-3}
$T$ proves that some forcing notion $P$ forces
$\psi^{\dot{H}_{\omega_1}}$;
\item \label{Thm:mainthm-3-1}
$T\vdash \psi^{H_{\omega_1}}$;
\item \label{Thm:mainthm-3-4}
$\psi\in T^*$.
\end{enumerate}
\end{Theorem}

\begin{Theorem}\label{Thm:mainthm-4}
Let $\sigma_{\omega,\NS_{\omega_1}}$ be
the extension of $\sigma_\omega$ with a unary predicate symbol $\NS_{\omega_1}$ and a 
constant symbol $\omega_1$.
Consider the $\bp{\in,\omega_1,\NS_{\omega_1}}$-sentences:
\[
\theta_{\omega_1}\equiv\omega_1\text{ is the first uncountable cardinal},
\]
\[
\theta_{\mathrm{Stat}}\equiv\forall x\,(x\subseteq\omega_1\text{ is non-stationary}\leftrightarrow \NS_{\omega_1}(x)).
\]
Let $\ZFC^*_{\omega,\NS_{\omega_1}}$ be the theory
\[
\ZFC^*_\omega+\theta_{\mathrm{Stat}}+\theta_{\omega_1}.
\]

Let also $\theta_{\bool{SC}}$ be the $\in$-sentence:
\[
\text{There are class many supercompact cardinals}.
\]

TFAE for any 
\[
T\supseteq \ZFC^*_{\omega,\NS_{\omega_1}}+\theta_{\bool{SC}}
\]
and for any
$\Pi_2$-sentence $\psi$ for $\sigma_{\omega,\NS_{\omega_1}}$:
\begin{enumerate}
\item \label{Thm:mainthm-4-2}
For all universal $\sigma_{\omega,\NS_{\omega_1}}$-sentences $\theta$
such that $T+\theta$ is consistent, so is $T_\forall+\theta+\psi$;
\item \label{Thm:mainthm-4-3}
$T$ proves that some forcing notion $P$ forces 
$\psi^{\dot{H}_{\omega_2}}$;
\item \label{Thm:mainthm-4-1}
$T_\forall+\ZFC^*_{\omega,\NS_{\omega_1}}+\theta_{\bool{SC}}+
\stUB\vdash \psi^{H_{\omega_2}}$.
\end{enumerate}
\end{Theorem}
See Remark \ref{Rem:keyrem}(\ref{Rem:keyrem-8}) for some information on 
$\stUB$.

In this article we will give a self-contained proof of Thm. \ref{Thm:mainthm-3} and of a weaker variation of 
Thm. \ref{Thm:mainthm-4} (cfr. Thm. \ref{Thm:mainthm-5}). 
Thm. \ref{Thm:mainthm-4} is an easy corollary of results which we will formulate in this paper at a later stage (since they need more terminology then what has been introduced so far to be properly stated), and which will be proved in a sequel of this paper (since their proof is considerably more involved, and its inclusion here would make the length of this paper grow exponentially).

Some of the following remarks are technical and require a
strong background in set theory. The reader can safely skip
them without compromising the comprehension of the remainder of this paper.

\begin{Remark}\label{Rem:keyrem}
\emph{}

\begin{enumerate}
\item The theories $T$ considered in all the above theorems are definable and conservative extensions of their
$\in$-fragment; more precisely: for any of the above signatures $\tau$ there is a recursive list of axioms $T_\tau\subseteq T$ such that any $\in$-structure admits a unique extension to a $\tau$-structure which models $T_\tau$ (see Fact \ref{fac:basicfact} below).

The key but trivial observation is that in the new signatures one can express the same concepts one can express in 
the signature $\in$, but using 
for many of these concepts formulae of much lower complexity according to the prenex normal form stratification.
For example:
\begin{itemize} 
\item
In $\sigma_\omega$ \emph{projective determinacy} is expressible by means of a family of contably many
atomic sentences (see item \ref{Rem:keyrem-5} below).
\item
\emph{There is an uncountable cardinal} is expressible by the $\Sigma_2$-sentence for $\tau_\ST$ (and 
$\sigma_\omega$)
\[
\exists x\, [(x\text{ is an ordinal }\wedge \omega\in x)\wedge
\forall f\,[(f \text{ is a function}\wedge\dom(f)\in x)\rightarrow \ran(f)\neq x]
\]
(and this concept \textbf{cannot} be expressed in this signature by a formula of lower complexity, even in $\sigma_\omega$).
\item
On the other hand the above sentence is $\ZFC^{*-}_{\omega,\NS_{\omega_1}}$-equivalent to the
universal $\tau_\ST\cup\bp{\omega_1}$-sentence:
\[
[(\omega_1\text{ is an ordinal }\wedge \omega\in \omega_1)\wedge
\forall f\,[(f \text{ is a function}\wedge\dom(f)\in \omega_1)\rightarrow \ran(f)\neq \omega_1].
\]
\end{itemize}
Our focus will be to understand which concepts are expressible by universal sentences and which 
are expressible by $\Pi_2$-sentences in the appropriate signatures.

One of the basic intuition leading to the above theorems is that the axiomatization of set theory
in the signature $\bp{\in}$ make unnecessarily complicated the formalization
of many basic set theoretic properties; however if one adds the ``right'' predicates 
and constant symbols to denote certain basic properties (i.e. the $\Delta_0$-properties) 
and certain more complicated ones of which we have however a clear grasp (i.e. the projective 
sets and the non-stationary ideal), the logical complexity of set-theoretic concepts lines 
up with our understanding of them. Once this operation 
is performed, the two theorems above show (assuming large cardinals) that 
for $\Pi_2$-properties consistency with the universal fragment of $T$ overlaps with forcibility over models of $T$
and with provability with respect to the right extension of $T$.

\item
Theorems 
\ref{Thm:mainthm-3} and \ref{Thm:mainthm-4}
are special instantiation of a method which pairs the notion of model companionship 
with generic absoluteness results. Roughly the equivalence between (1) and (3) of the two
theorems follow from the existence of a model companion for $T$ in the appropriate signature, while the
equivalence of (2) and (3) follows from generic absoluteness results.

\item The reader may wonder why Thm. \ref{Thm:mainthm-3} does not conflict with G\"odel's incompleteness theorem. 
Let $T_0$ be the theory 
\[
\ZFC_\omega+\text{ there are class many Woodin cardinals}.
\] 
The  G\"odel $\in$-sentence  $\mathrm{Con}(T_0)$ and its negation
become atomic $\tau_\ST$-sentences (since all their quantifiers range over $\omega$), hence a part of the universal 
(or of the $\Pi_2$) theory of any complete extension of $T_0$.
However there are complete extensions of $T_0$ containing $\mathrm{Con}(T_0)$ 
and others containing its negation, therefore the equivalences set forth in Thm. \ref{Thm:mainthm-3} are not violated letting
the $\psi$ of the theorem be $\mathrm{Con}(T_0)$.

Note on the other hand that the content of Thm. \ref{Thm:mainthm-3} is that 
``almost'' any question of second order arithmetic (see the next item) is decided by 
large cardinal axioms: apart from G\"odel sentences, 
it is clearly open
whether there are more interesting arithmetic (or even projective) statements
(such as Golbach's conjecture or Schanuel's conjecture or Riemann's hypothesis) 
which are independent 
of these axioms.
Similar considerations apply to Theorem \ref{Thm:mainthm-4}.

\item \label{Rem:keyrem-4}
Every lightface projective set (i.e. any definable subset without parameters of the structure
$(\pow{\mathbb{N}},\in,\subseteq)$) is the extension of a quantifier free formula in $\sigma_\omega$.
Letting $\phi_n(x,y)$ by a $\tau_\ST$-formula defining a universal set for $\Sigma^1_{n+1}$-sets,
it is not hard to see that
projective determinacy  (according to the notation of \cite[Section 20.A]{kechris:descriptive}) 
is given by an axiom scheme of $\tau_\ST$-sentences in which quantifiers range just over subsets of $\pow{\omega^{<\omega}}$.
In particular projective determinacy is expressed by a family of  atomic sentences for $\sigma_\omega$ 
in $\ZFC^*_\omega$.
%

\item \label{Rem:keyrem-5}
The negation of the Continuum hypothesis 
$\mathsf{CH}$ is expressible in the signature 
$\tau_\ST\cup\bp{\omega_1}\subseteq\sigma_{\omega,\NS_{\omega_1}}$ 
as the $\Pi_2$-sentence $\psi_{\neg\mathsf{CH}}$:
\begin{align*}
(\omega_1\text{ is the first uncountable cardinal})\wedge\\ 
\wedge \forall f\,
[(f\text{ is a function }\wedge \dom(f)=\omega_1)\rightarrow\exists r (r\subseteq\omega\wedge r\not\in \ran(f))].
\end{align*}
Most of third order number theory is expressible in this signature by a $\Pi_2$-sentence, 
for example this is the case for \emph{Suslin's hypothesis}, 
\emph{every Aronszjain tree is special}, and a variety of other statements.

%

\item \label{Rem:keyrem-8}
It is out of the scopes of the present paper to define 
$(*)$-$\UB$; it will be essentially used only in the sequel of this work;
\ref{Thm:mainthm-4-1} of Thm. \ref{Thm:mainthm-4} is
the unique place of this paper where this statement will ever be mentioned.
For the convenience of the interested reader we include its definition in Section~\ref{sec:furtherresults}.
Let us just briefly say that 
$(*)$-$\UB$ is the strong form of Woodin's axiom $(*)$ asserting that $\NS_{\omega_1}$ is saturated together with
the existence of an $L(\UB)$-generic filter for Woodin's 
$\Pmax$-forcing\footnote{See \cite{HSTLARSON} for details on $\Pmax$.} 
(where $L(\UB)$ is the smallest transitive model of $\ZF$ containing all the universally Baire sets).
\end{enumerate}
\end{Remark}

Our ambition is to make the remainder of this paper self-contained and 
accessible to any person who has a fair acquaintance with set theory
and first order logic. From now on no familiarity with forcing, large cardinal axioms, forcing axioms 
is needed or assumed on the reader, all it is required is just to accept as meaningful 
the statement of these theorems.

The following piece of notation will be used.
\begin{Notation}\label{not:keynotation}
\emph{}

\begin{itemize}
\item
$\tau_{\mathsf{ST}}$ is the extension of the first order signature $\bp{\in}$ for set theory 
which is obtained 
by adjoining predicate symbols
$R_\phi$ of arity $n$ for any $\Delta_0$-formula $\phi(x_1,\dots,x_n)$, function symbols of arity $k$
for any $\Delta_0$-formula $\theta(y,x_1,\dots,x_k)$
and constant symbols for 
$\omega$ and $\emptyset$.
\item
$\sigma_\ST$ is the signature containing a 
predicate symbol 
$S_\phi$ of arity $n$ for any $\tau_\ST$-formula $\phi$ with $n$-many free variables.
\item
$\sigma_\kappa=\sigma_\ST\cup\tau_\ST\cup\bp{\kappa}$ with $\kappa$ a constant symbol.
\end{itemize}

\begin{itemize}
\item $\ZFC^{-}$ is the
$\in$-theory given by the axioms of
$\ZFC$ minus the power-set axiom.

\item
$T_\ST$ is the $\tau_\ST$-theory given by the axioms
\[
\forall \vec{x} \,(R_{\forall x\in y\phi}(y,\vec{x})\leftrightarrow \forall x(x\in y\rightarrow R_\phi(y,x,\vec{x}))
\]
\[
\forall \vec{x} \,[R_{\phi\wedge\psi}(\vec{x})\leftrightarrow (R_{\phi}(\vec{x})\wedge R_{\psi}(\vec{x}))]
\]
\[
\forall \vec{x} \,[R_{\neg\phi}(\vec{x})\leftrightarrow \neg R_{\phi}(\vec{x})]
\]
\[
(\forall \vec{x}\exists!y \,R_{\phi}(y,\vec{x}))\rightarrow (\forall \vec{x}\,R_{\phi}(f_\phi(\vec{x}),\vec{x}))
\]
for all $\Delta_0$-formulae $\phi(\vec{x})$, together with the $\Delta_0$-sentences
\[
\forall x\in\emptyset\,\neg(x=x),
\]
\[
\omega\text{ is the first infinite ordinal}
\]
(the former is an atomic $\tau_\ST$-sentence, the latter is expressible as the $\Pi_1$-sentence for 
$\tau_\ST$ stating that
$\omega$ is a non-empty limit ordinal contained in any other non-empty limit ordinal).
\item
$T_\kappa$ is the $\sigma_\ST\cup\bp{\kappa}$-theory given by the axioms
\begin{equation}\label{eqn:keytau*kappa}
\forall x_1\dots x_n\,[S_\psi(x_1,\dots,x_n)\leftrightarrow 
(\bigwedge_{i=1}^n x_i\subseteq \kappa^{<\omega}\wedge \psi^{\pow{\kappa^{<\omega}}}(x_1,\dots,x_n))]
\end{equation}
as $\psi$ ranges over the $\in$-formulae.
\item
$\ZFC^-_\ST$ is the $\tau_\ST$-theory 
\[
\ZFC^{-}\cup T_\ST
\] 
\item
$\ZFC^-_\kappa$ is the $\tau_\ST\cup\bp{\kappa}$-theory 
\[
\ZFC^-_\ST\cup\bp{\kappa\text{ is an infinite cardinal}};
\]
\item
$\ZFC^{*-}_\kappa$ is the $\sigma_\kappa$-theory 
\[
\ZFC^-_\kappa\cup T_\kappa;
\]
\item 
$\ZFC^{*-}_\omega$ is 
\[
\ZFC^{*-}_\kappa\cup\bp{\kappa\text{ is the first infinite cardinal}};
\] 
\item
Accordingly we define $\ZFC_\ST$,
$\ZFC_\kappa$, $\ZFC^*_\ST$,
$\ZFC^*_\kappa$,
 $\ZFC^*_\omega$.
\end{itemize}
\end{Notation}

\begin{Fact}\label{fac:basicfact}
Every $\sigma_\kappa$-formula is $T_\kappa\cup T_\ST$-equivalent to an $\bp{\in,\kappa}$-formula.

Moreover assume  $\kappa$ is a definable cardinal (i.e. $\kappa=\omega$ or $\kappa=\omega_1$); 
more precisely assume there is an $\in$-formula $\psi_\kappa(x)$ such that
\[
\ZFC^-\vdash\exists!x\,[\psi_\kappa(x)\wedge (x\text{ is a cardinal})].
\]

Then every $\sigma_\kappa$-formula is 
$\ZFC^{*-}_\kappa+\psi_\kappa(\kappa)$-equivalent to an $\in$-formula.
\end{Fact}
\begin{proof}
The axioms of $T_\ST$ and $T_\kappa$ are cooked up exactly so that one can prove the result by a straightforward induction on the $\sigma_\kappa$-formulae 
(see also the proof of Prop. \ref{prop:quantelimallthe}).
\end{proof}

Theorem \ref{Thm:mainthm-3} 
is an immediate corollary of Woodin's generic results for second order number 
theory (cfr. \cite{WOOBOOK})
coupled with the following theorem:

\begin{Theorem} \label{Thm:mainthm-1}
Assume $T\supseteq \ZFC^*_\kappa$
is a $\sigma_\kappa$-theory.
Then $T$ has a model companion $T^*$.
Moreover for any $\Pi_2$-sentence $\psi$ for $\sigma_\kappa$, TFAE:
\begin{enumerate}
\item $\psi\in T^*$;
\item
$T\vdash\psi^{H_{\kappa^+}}$;
\item
For all universal $\sigma_\kappa$-sentences $\theta$,
$T_\forall+\theta$ is consistent if and only if so is $T_\forall+\theta+\psi$.
\end{enumerate}
\end{Theorem}

We note that approximations to 
Thm. \ref{Thm:mainthm-1} for the case $\kappa=\omega$, and to
Thm.~\ref{Thm:mainthm-3} 
already appears in \cite{VIAVENMODCOMP}.

The present paper give a self-contained proof of Theorems
 \ref{Thm:mainthm-3} and \ref{Thm:mainthm-1}. We defer to a second paper the proof of Theorem \ref{Thm:mainthm-4} 
 (which reposes on the recent breakthrough by Asper\`o and Schindler
establishing that Woodin's axiom $(*)$ follows from $\MM^{++}$ \cite{ASPSCH(*)}); here we will prove a weaker version 
of it (cfr. Thm. \ref{Thm:mainthm-5}) at the end of Section \ref{sec:Hkappa+}.

We prove rightaway Thm. \ref{Thm:mainthm-3} assuming Thm. \ref{Thm:mainthm-1}:
\begin{proof}
Woodin's generic absoluteness results for second order number theory give that
\ref{Thm:mainthm-3}(\ref{Thm:mainthm-3-1}) and 
\ref{Thm:mainthm-3}(\ref{Thm:mainthm-3-3})
are equivalent
(we give here a self-contained proof of this particular instance of Woodin's results in
Theorem \ref{thm:genabshomega1}). 
Theorem \ref{Thm:mainthm-1} gives the equivalence of 
\ref{Thm:mainthm-3}(\ref{Thm:mainthm-3-2}) and 
\ref{Thm:mainthm-3}(\ref{Thm:mainthm-3-1}).
\end{proof}

The proof and statement of Thm. \ref{Thm:mainthm-1}
require familiarity with set theory at the level of an undergraduate book (for example \cite{JECHHRBACEKBOOK} coupled with \cite[Chapters III, IV]{KUNEN} is far more than sufficient)
as well as familiarity 
 with the notion of model companionship. 
 
 To complete this introductory section it is
 convenient to sort out how the
definable extensions $\ZFC^*_\ST$, $\ZFC^*_\omega$,
$\ZFC^*_{\omega,\NS_{\omega_1}}$ behave with 
respect to forcing. A central role is played by large cardinal axioms.
The reader can safely skip this remark without compromising the reading of 
the sequel of this paper.  

 \begin{Remark}\label{rmk:keyrembis}
We outline here the invariance under forcing of the $\Pi_1$-theory of $V$ in 
certain natural signatures; since the universal fragment of a theory $T$
determines completely its model companion, the fact that in certain signatures 
$\tau$ forcing cannot change the $\Pi_1$-theory of $V$
(in combination with Levy's absoluteness theorem) is the key to understand 
why set theory can have a model companion in some of these signatures, and 
why the properties of the model companion 
theory are paired with generic absoluteness results.

\begin{itemize}
\item
The standard absoluteness results of Kunen's 
book \cite[Ch. IV]{KUNEN} show that if $G$ is $V$ 
generic for some forcing notion $P\in V$,
$V\sqsubseteq V[G]$ for $\tau_\ST$.
\item
Shoenfield's absoluteness Lemma entails that if $G$ is $V$ 
generic for some forcing notion $P\in V$,
$V\prec_1 V[G]$ for $\tau_\ST$.

This holds since $H_{\omega_1}\prec V$ and 
$H_{\omega_1}^{V[G]}\prec V[G]$ (cfr. Lemma \ref{lem:levyabsHkappa+}), and $H_{\omega_1}\prec H_{\omega_1}^{V[G]}$ 
(see for example \cite[Lemma 1.2]{VIAMMREV}) for the signature $\tau_\ST$.
\item
Major results of the Cabal seminar bring that assuming the existence of 
class many Woodin cardinals in $V$, if $G$ is $V$ 
generic for some forcing notion $P\in V$,
$V\sqsubseteq V[G]$ for $\sigma_\omega$ (roughly because 
$H_{\omega_1}^V\prec H_{\omega_1}^{V[G]}$ by 
Thm. \ref{thm:genabshomega1}, while 
$H_{\omega_1}^V\prec_1 V$ and $H_{\omega_1}^{V[G]}\prec_1 V[G]$ by 
Lemma \ref{lem:levyabsHkappa+}) for the signature $\sigma_\omega$.
More generally the same large cardinal assumptions and argument yield that 
$V\sqsubseteq V[G]$ also for the signature extending $\tau_\ST\cup\UB$ 
with predicate symbols for all universally Baire sets of $V$ 
(instead of considering just the lightface projective sets as done by 
$\sigma_\omega$).
\item
Assume $G$ is $V$ 
generic for some forcing notion $P\in V$.
$V\sqsubseteq V[G]$ for $\tau_\ST\cup\bp{\omega_1,\NS_{\omega_1}}$ if and only if $P$ is stationary set preserving: for the atomic predicates $\NS_{\omega_1}$
the formula $\neg\NS_{\omega_1}(S)$ is preserved between
$V$ and $V[G]$ for all $S\subseteq\omega_1$ in $V$ only
in this case. The sentence \emph{$\omega_1$ is the first uncountable cardinal} is preserved only if
$P$ does not collapse $\omega_1$.
\item
Assuming the existence of 
class many Woodin cardinals in $V$ for \emph{any}
forcing $P\in V$ (i.e. also if $P$ is not stationary set preserving or collapses $\omega_1$),
for any $G$ $V$-generic for $P$, 
$V[G]$ and $V$ satisfy the same 
$\Pi_1$-sentences for $\sigma_{\omega,\NS_{\omega_1}}$
(Thm. \ref{Thm:mainthm-8}).
\item
On the other hand the signature $\sigma_\kappa$ with $\kappa\geq\omega_1$ behaves
badly with respect to forcing; one has to put severe limitation on the type of forcings $P$
considered in order
to maintain that $V\sqsubseteq V[G]$ or just that $V$ and $V[G]$ satisfy the same universal 
$\sigma_\kappa$-sentences
(see Remark \ref{Rem:keyrem1} to appreciate the difficulties). However we will prove an interesting variation of Thm.
\ref{Thm:mainthm-4} for $\sigma_\kappa$ in case $\kappa$ is interpreted by $\omega_1$ 
(cfr. Thm. \ref{Thm:mainthm-5}).
\end{itemize}

These results combined together give the following argument for the proof of
(\ref{Thm:mainthm-4-3}) implies (\ref{Thm:mainthm-4-2}) of Thm. \ref{Thm:mainthm-4}
(mutatis mutandis for the proof of (\ref{Thm:mainthm-3-3}) implies (\ref{Thm:mainthm-3-2}) of Thm. \ref{Thm:mainthm-3}): 
let $\psi $ be a $\Pi_2$-sentece for 
$\sigma_{\omega,\NS_{\omega_1}}$ satisfying (\ref{Thm:mainthm-4-3}).
Given some $\Pi_1$-sentence 
$\theta$ for $\sigma_{\omega,\NS_{\omega_1}}$ consistent with $T$, find $\mathcal{M}$ model of $T+\theta$.
By (\ref{Thm:mainthm-4-3}) some forcing $P\in\mathcal{M}$
forces $\psi^{H_{\omega_2}}$. By Thm. \ref{Thm:mainthm-8} and Levy's absoluteness Lemma
\ref{lem:levyabsHkappa+}, the theory
$T_\forall+\theta+\psi$ holds in $H_{\omega_2}^{\mathcal{N}}$ whenever $\mathcal{N}$ is a generic extension of 
$\mathcal{M}$ by $P$.
\end{Remark}

 The paper is organized as follows:
 \begin{itemize}
 \item
 Section \ref{sec:Hkappa+} proves Thm. \ref{Thm:mainthm-1}
 (WARNING: familiarity with the notion of model companionship is required).
 We also include in its last part a proof of a weaker variation of Thm. \ref{Thm:mainthm-4} (cfr. Thm. \ref{Thm:mainthm-5}).
 \item
 Section \ref{sec:modth} gives a detailed account of model completeness and model companionship\footnote{
 Our ambition is that this section could serve as a compact self-contained account
 of the key properties of model companion theories.}.  
 \item
 Section \ref{sec:auxres} gives a self-contained proof of the form of Levy absoluteness and of the particular form of Woodin's generic absoluteness results we employ in this paper\footnote{We included these results 
 here, because the versions of these results we found in the literature were not exactly fitting to our set up. Again our purpose for this section is to simplify the reader's task, as well to give minor improvements of known results.}.
 \item Section \ref{sec:somecomm} gives some intuitions motivating 
 Theorems \ref{Thm:mainthm-3},
 \ref{Thm:mainthm-4}, \ref{Thm:mainthm-1}, and
 a few ``philosophical'' 
considerations we can draw from them (in particular an argument for the failure of $\bool{CH}$).
The reader can safely skip it without compromising the comprehension of the remainder of the paper (WARNING: familiarity with the notion of model companionship is required).
\item Section \ref{sec:furtherresults} collects the main results we will prove in a sequel of this paper.
 \end{itemize}
 
 The paper contains (overly?) detailed proofs of every non-trivial result 
(many of which can be also found elsewhere i.e. most ---if not all--- of those appearing in 
sections \ref{sec:modth} and \ref{sec:auxres}), 
this has been made at the expenses of its brevity.
Our hope is that this approach makes the paper accessible to all scholars with a basic knowledge of set theory and model theory.
 
The reader unfamiliar with the notion of model companionship and its main implications should start with Section
\ref{sec:modth}, rather than with Sections \ref{sec:Hkappa+}
or \ref{sec:somecomm}.

{\small
\subsection*{Acknowledgements}

This research has been completed while visiting the \'Equipe de Logique 
Math\'ematique of the IMJ in 
Paris 7 in the fall semester of 2019.
The author thanks Boban Veli\v{c}kovi\'c, David Asper\'o, Giorgio Venturi, 
for the fruitful discussions held on the topics
of the present paper; I particularly thank Venturi for contributing substantially to 
the elaboration of many of the considerations in Section \ref{sec:somecomm}, and 
Veli\v{c}kovi\`c for providing counterexamples to many of my attempts to 
produce generalizations of the results of the present paper. 

The opportunity to present preliminary versions of these results in the set theory seminar of the
\'Equipe has also given me the possibility to improve them substantially. 
I thank all the people attending it for their many useful comments, in particular Alessandro Vignati.

There are many others with whom I exchanged frutiful and informative discussions on these topics, among them Philipp Schlicht and Neil Barton.
}

\section{Some comments}\label{sec:somecomm}

Let us bring to light some ideas bringing to Theorems \ref{Thm:mainthm-3}, \ref{Thm:mainthm-4}, 
\ref{Thm:mainthm-1}. 

\subsubsection*{Correct signatures for set theory}
A first basic idea  is
that bounded formulae express ``simple'' properties of sets. The Levy stratification of set-theoretic properties consider those expressed by bounded formulae the simplest; then the complexity increases as unbounded quantifiers lines up in the prenex normal form of a formula. In particular the Levy stratification
matches exactly with the stratification of $\tau_\ST$-formulae according to the number of alternating quantifiers
in their $\ZFC_\ST$-equivalent prenex form.

Assume instead we measure the complexity of a set theoretic property $P$ according to the number of alternating quantifiers of the prenex normal form of its 
$\in$-formalization. 
Then many basic properties already have high complexity: the formula
$z=\bp{x,y}$ is expressed by a $\Pi_1$ formula for $\in$; the $\in$-formula expressing
\emph{$f$ is a function} by means of Kuratowski pairs to define relations 
has already so many quantifiers that one cannot  estimate their numbers at first glance, etc.
If we resort to the axiomatization of set theory given by $\ZFC_\ST$, this problem is overruled, and these two properties
are expressed by atomic $\tau_\ST$-formulae\footnote{There are atomic 
$\tau_\ST$-formulae whose $\ZFC_{\ST}$-equivalent prenex $\in$-formula of least complexity 
has an arbitrarily large number of alternating quantifiers.}.
In particular reformulating $\ZFC$ using the signature $\tau_\ST$ recalibrates the complexity of formulae
letting arbitrarily complex $\in$-formulae become atomic, while not changing the set of 
$\ZFC$-provable theorems, and stratifies set theoretic properties 
in complete accordance with the Levy hierarchy\footnote{Nonetheless there are $\in$-sentences 
whose least complexity $\ZFC_\ST$-equivalent
$\tau_\ST$-sentence in prenex normal form 
has an arbitrary finite number of alternating quantifiers, examples are given
by lightface definable universal sets for $\Sigma^1_n$-properties (cfr. \cite[Thm. 4.6]{VIAVENMODCOMP}).}.


\subsubsection*{Levy absoluteness and model companionship}
Mostowski collapsing theorem and the axiom of choice allow to code a set by a well founded relation on its hereditary cardinality, and
in this way translate in an ``absolute manner'' questions about sets in $H_{\kappa^+}$ to questions about
$\pow{\kappa}$. The content of Theorem~\ref{Thm:mainthm-1} is that  we can give a very precise model-theoretic meaning to the term ``absolute manner'': any $\sigma_\kappa$-formula is 
$\ZFC^-_\kappa+\bp{\emph{all sets have size $\kappa$}}$-equivalent
to a universal $\sigma_\kappa$-formula and to an existential $\sigma_\kappa$-formula, i.e. it is a provably $\Delta_1$-property in this theory. What happens is that we encoded complicated questions
about the power-set of $\kappa$ by means of atomic predicates, since the axioms listed in
\ref{eqn:keytau*kappa} amount to a method to eliminate quantifiers ranging over $\pow{\kappa}$.
So Theorem~\ref{Thm:mainthm-1} is another way to reformulate that
 the first order theory of $H_{\kappa^+}$ reduces to
the first order theory of $\pow{\kappa}$ in an absolute manner.

Remark also that for all models $(V,\in)$ of $\ZFC$ and all cardinals 
$\kappa\in V$ and all signatures 
\[
\tau_\kappa\subseteq\tau\subseteq \tau_\kappa\cup\pow{\pow{\kappa}}
\]
$(H_{\kappa^+}^V,\tau^V)\prec_1(V,\tau^V)$ is the unique transitive 
substructure of 
$V$ containing $\pow{\kappa}$
which models $\ZFC^-$ and the 
$\Pi_2$-sentence for $\tau_\kappa$
\[
\forall X\exists f\,(f:\kappa\to X\text{ is surjective}).
\]
In particular if a model companion of the $\tau$-theory
of $V$ exists, this can only be the $\tau$-theory of 
$H_{\kappa^+}$.

\subsubsection*{Generic invariance of the $\Pi_1$-theory of $V$ in a given signature}
We say that a 
signature $\sigma$ is generically tame for a $\sigma$-theory $T$ extending $\ZFC_\kappa$
if the $\Pi_1$-consequences of $T$ must be preserved through forcing 
extensions of models of $T$ (which brings the implication (\ref{Thm:mainthm-3-3})$\to$(\ref{Thm:mainthm-3-2}) of Theorem \ref{Thm:mainthm-3} --- as well as the corresponding implications of Theorems 
\ref{Thm:mainthm-4}, \ref{Thm:mainthm-7}, \ref{Thm:mainthm-6} --- by the argument sketched in Remark \ref{rmk:keyrembis}).

Theorem \ref{Thm:mainthm-8} shows that this generic invariance holds 
for all\footnote{See Notation \ref{not:keynotation-2} for the definition of
$\tau_{\NS_{\omega_1}}$} $\sigma\subseteq\tau_{\NS_{\omega_1}}\cup\UB^V$ 
where $\UB^V$ denotes the family of
universally Baire sets of some $(V,\in)$ which models $\ZFC+$\emph{large cardinals}.

Theorem \ref{Thm:mainthm-8} is close to optimal:
a (for me surprising) fact remarked by Boban Veli\v{c}kovi\`c is that Thm. \ref{Thm:mainthm-8} cannot possibly hold for any $\sigma\supseteq \tau_{\ST}\cup\bp{\omega_1,\omega_2}$, where $\omega_2$ is a constant which names the second uncountable cardinal:
\begin{quote}
$\Box_{\omega_2}$ is a $\Sigma_1$-statement for $\tau_{\omega_2}=\tau_{\ST}\cup\bp{\omega_1}\cup\bp{\omega_2}$:
\begin{align*}
\exists\bp{C_\alpha:\alpha<\omega_2}&[\\
&\forall \alpha\in\omega_2\, (C_\alpha\text{ is a club subset of }\alpha)\wedge\\
&\wedge\forall\alpha\in\beta\in\omega_2 \,(\alpha\in \lim(C_\beta)\rightarrow C_\alpha=C_\beta\cap\alpha)\wedge\\
&\wedge\forall\alpha\,( \otp(C_\alpha)\leq\omega_1)\\
&].
\end{align*}
\end{quote}
$\Box_{\omega_2}$ is forcible by very nice forcings (countably directed and $<\omega_2$-strategically closed), 
and its negation is forcible by $\Coll(\omega_1,<\delta)$ whenever $\delta$ is supercompact.

In particular the $\Pi_1$-theory for $\tau_{\omega_2}$ of any forcing extension $V[G]$ of $V$ can be 
destroyed in a further forcing extension $V[G][H]$, hence is not invariant across forcing extensions of $V$ in any possible sense, assuming large cardinals in $V$.

Theorems \ref{Thm:mainthm-3}, \ref{Thm:mainthm-4}, \ref{Thm:mainthm-7}, \ref{Thm:mainthm-6} show that the strong form of consistency given by (\ref{Thm:mainthm-3-2}) of Theorem \ref{Thm:mainthm-3} can characterize forcibility 
(at least for $\Pi_2$-sentences in the appropriate signature), if large cardinals enter the picture.


\subsubsection*{Model companionship and generic absoluteness}
The first order theory of $\pow{\kappa}$ for any infinite $\kappa$ is very sensitive to forcing; 
but this depends on two parameters: whether or not we assume large cardinals, 
and what is the signature in which we look at the first order theory of $\pow{\kappa}$.

Theorem \ref{Thm:mainthm-8}
shows that we can 
``tune'' the signature $\sigma$ so that for any $\sigma$-theory $T$ extending $\ZFC+$\emph{large cardinals}: 
\begin{itemize}
\item
the signature is expressive (i.e. many questions of second or third order arithmetic can be encoded by
simple sentences, i.e. $\Pi_2$-sentences for $\sigma$);
\item the signature is not too expressive (i.e. the questions of second or third order arithmetic whose truth value
can be changed by means of forcing cannot be encoded by $\Pi_1$-sentences for $\tau$; in particular the $\Pi_1$-fragment of $\ZFC$ in the new signature is invariant across the generic multiverse, cfr. Thm. \ref{Thm:mainthm-8}).
\end{itemize}
These two conditions entail that \ref{Thm:mainthm-3}(\ref{Thm:mainthm-3-3}) implies 
\ref{Thm:mainthm-3}(\ref{Thm:mainthm-3-2}) (respectively \ref{Thm:mainthm-4}(\ref{Thm:mainthm-4-3}) implies \ref{Thm:mainthm-4}(\ref{Thm:mainthm-4-2})).
Generic absoluteness results give that
\ref{Thm:mainthm-3}(\ref{Thm:mainthm-3-3}) is equivalent to 
\ref{Thm:mainthm-3}(\ref{Thm:mainthm-3-1}) (respectively \ref{Thm:mainthm-4}(\ref{Thm:mainthm-4-3}) is equivalent \ref{Thm:mainthm-4}(\ref{Thm:mainthm-4-1})).

Model completeness of the relevant theories gives the missing implication from (1) to (3) of 
Theorems \ref{Thm:mainthm-3}, \ref{Thm:mainthm-4}, \ref{Thm:mainthm-7}, \ref{Thm:mainthm-6}.

\subsubsection*{Model companionship and generic absoluteness for second order number theory}
The standard argument used in set theory to assert that $\Delta_0$-properties are simple, is 
their invariance between transitive models, which in turns imply that their 
truth values cannot be changed by means of forcing.

Now consider second order number theory i.e.: the theory of the structure
$(\pow{\omega},\in)$; modulo the by-interpretation which identifies a hereditarily countable set
with the graph of the transitive closure of its singleton (see Section \ref{sec:Hkappa+}), the theory of 
$(\pow{\omega},\in)$
has the same set of theorems as the
first order theory of the structure $(H_{\omega_1},\tau_{\ST})$, which in turns
(by Fact \ref{fac:basicfact}) has the same set of theorems as the structure $(H_{\omega_1},\sigma_{\omega})$.
The first order theory of $H_{\omega_1}$ in any of these signatures
can vary (by means of forcing) if one denies the existence of large cardinals (for example there can be lightface definable projective well-orders, or not): on the other hand 
a major result of Woodin 
is that assuming large cardinals, the first order theory of $(H_{\omega_1},\in)$
is invariant with respect to forcing. 
The equivalence of (\ref{Thm:mainthm-3-3}) and (\ref{Thm:mainthm-3-2}) in Theorem \ref{Thm:mainthm-3}
says that this theory is fixed by any reasonable method
to produce its models, not just forcing.

Now we combine these results with the clear picture given by projective determinacy of the theory of projective sets:
much in the same way we accept bounded formulae as ``simple'' predicates and make them equivalent to
atomic formulae by means of $\ZFC_{\ST}$, if we accept as true large cardinal axioms, we are forced to consider projective sets of reals as ``simple'' predicates; $\ZFC_\omega$
includes them among the atomic predicates. 
Once we do so the first order theory
of $H_{\omega_1}$ is ``tame'' i.e. model complete, hence
it realizes all $\Pi_2$-sentences which are 
consistent with its universal fragment (cfr. Fact \ref{fac:charKaihull}); moreover
large cardinals make provably true many of these $\Pi_2$-sentences, for example projective determinacy.

\subsubsection*{Model companionship and generic absoluteness for the theory of $\pow{\omega_1}$}
Theorem \ref{Thm:mainthm-4} extends the above considerations to the signature 
$\sigma_{\omega,\NS_{\omega_1}}$. In this case  a theory $T$ extending $\ZFC+$\emph{large cardinals} 
is just able to say that:
\begin{itemize}
\item
a $\Pi_2$-sentence for $\sigma_{\omega,\NS_{\omega_1}}$ is
consistent with the universal fragment of $T$  if it is $T$-provably forcible
(cfr.  \ref{Thm:mainthm-4-3} implies 
\ref{Thm:mainthm-4-2} of Thm. \ref{Thm:mainthm-4}, see Remark \ref{rmk:keyrembis} for a proof).
\item
The theory $T^*$ given by all $\Pi_2$-sentences 
$\psi$ for $\sigma_{\omega,\NS_{\omega_1}}$ 
such that $\psi^{H_{\omega_2}}$ is provably forcible is consistent (cfr.  \ref{Thm:mainthm-4-3} implies \ref{Thm:mainthm-4-1} of Thm. \ref{Thm:mainthm-4}, one of the main results of Woodin on $\Pmax$ \cite[Thm. 7.3]{HSTLARSON}).
\item
Recently Asper\`o and Schindler proved that $\MM^{++}$ implies $(*)$-$\UB$ \cite{ASPSCH(*)}. An immediate corollary of their 
result is that $T^*$ holds in the $H_{\omega_2}$ of models of $\MM^{++}$. This result allow to prove the missing implications in Thm. \ref{Thm:mainthm-4} (or in Thm. \ref{Thm:mainthm-7}, \ref{Thm:mainthm-6}).
\end{itemize}
We will prove the assertions in the above items using
Thm. \ref{Thm:mainthm-4}, \ref{Thm:mainthm-7}, \ref{Thm:mainthm-6} 
and Asper\`o and Schindler's result in a sequel of this paper.


\subsubsection*{Model completeness and bounded forcing axioms}
Let us now spend some more words relating model completeness to 
bounded forcing axioms and $\stUB$.
 Model companionship and model completeness capture in a model theoretic property the notion of ``generic'' structure for the models of a theory; this notion is 
 recurrent in various domains (not only restricted to model theory), we 
 mention two occurring in model theory:  in many cases the 
 Fraisse limit of a given family $\mathcal{F}$ of finite(ly generated) structures for a signature $\tau$ 
 is generic for the structures in $\mathcal{F}$; the algebraically closed field 
 are generic with respect to the class of fields.
Generic structures of a universal theory $T$
 realize as many $\Pi_2$-properties as it is consistently possible while remaining a model of $T$.  
The standard examples of generic structures for a first order theory $T$ are given by 
 $T$-existentially closed model, i.e. models which are 
$\Sigma_1$-substructures of any superstructure which realizes 
(the universal fragment of) $T$.
We will make this rigorous in
 Section \ref{sec:modth}.  
 
 Compare these observations with the formulation of bounded forcing axioms as principles of generic absoluteness
 (as done by Bagaria in \cite{BAG00}) stating that $H_{\omega_2}^V$ is a $\Sigma_1$-substructure of any generic extension of $V$ obtained by forcings in the appropriate class. 
 
 In essence Theorem \ref{Thm:mainthm-4} and Thm. \ref{Thm:mainthm-7} outline that forcing 
 axioms provide means to produce models of $H_{\omega_2}$ which are existentially closed 
 for their universal theory
 and realize as many $\Pi_2$-sentences as the iteration theorems producing them makes possible.

\subsubsection*{Why  $\mathsf{CH}$ is false}
Summing up on the above considerations, we believe we can give a strong argument against $\mathsf{CH}$:
\begin{quote}
Assume we adopt the stance that:
\begin{itemize} 
\item
Large cardinal axioms are true.
\item
We consider set theory as formalized by a \emph{definable extension} $T$ of $\ZFC+$\emph{large cardinals}
in a signature $\sigma$
where $\mathsf{CH}$ can be correctly formalized, i.e. $T$ is a definable extension of $\ZFC_{\omega_1}$
in the signature $\sigma\supseteq\tau_{\ST}\cup\bp{\omega_1}$ including a constant symbol for the first uncountable cardinal, so that:
\begin{itemize} 
\item
$\neg\mathsf{CH}$ is formalized by a $\Pi_2$-sentence for $\tau_{\ST}\cup\bp{\omega_1}$, 
(cfr. Remark \ref{Rem:keyrem}(\ref{Rem:keyrem-5})).
\item
The $\Pi_1$-fragment of $T$ is invariant across forcing extensions (so that the basic facts about 
$\pow{\omega_1}$ ---i.e those expressible by $\Pi_1$-sentences for $\sigma$--- are not changed by means of forcing (cfr.
$(\ref{Thm:mainthm-4-3})$ implies $(\ref{Thm:mainthm-4-2})$ of Thm. \ref{Thm:mainthm-4} holds for $T$).
\end{itemize} 
\end{itemize}

Furthermore to select which among all possible $T$ in the signature $\sigma$ 
gives the true ``axiomatization'' of set theory, 
we adopt the following criteria:
\begin{itemize}
\item
$T$
should maximize
the family of $\Pi_2$-sentences for $\sigma$ which are consistent with its $\Pi_1$-consequences (cfr. Thm. \ref{Thm:mainthm-4}(\ref{Thm:mainthm-4-2}));
\item 
there should be a simple and manageable axiom system for $T$ (cfr. Thm. \ref{Thm:mainthm-4}(\ref{Thm:mainthm-4-1}) or even $T$ has a model companion $T^*$).
\end{itemize}
With these premises, we conclude that Theorem \ref{Thm:mainthm-4}
(also with Thm. \ref{Thm:mainthm-7}, \ref{Thm:mainthm-6})
implies that
$\mathsf{CH}$ is false (since $\neg\mathsf{CH}$ is provably forcible from $T$).
\end{quote}

We can further reinforce our case by remarking
that: 
\begin{itemize}
\item The same assumptions on $T$ and $T^*$ entail
$2^{\aleph_0}=\aleph_2$ holds
in any $\tau$-model $\mathcal{N}$ of $T_\forall+\ZFC$ 
in which $T^*$ holds in $H_{\omega_2}^{\mathcal{N}}$:
$2^{\aleph_0}=\aleph_2$ is not a $\Pi_2$-sentence for $\tau_{\omega_1}$, 
but it is a consequence of $\Pi_2$-sentences for $\tau_{\ST}\cup\bp{\omega_1}$ which hold assuming $\mathsf{BPFA}$.
One such sentence is given by Caicedo and Velickovic in
\cite{CAIVEL06}:
\[
\forall \mathcal{C}\text{ ladder system on $\omega_1$}\,\forall r\subseteq\omega\,
\exists \alpha\,\exists f\;[(f:\omega_1\to \alpha\text{ is surjective})\wedge\psi(\mathcal{C},r,\alpha)]
\]
where $\psi(x,y,z)$ is a $\Sigma_1$-formula for $\tau_{\ST}\cup\bp{\omega_1}$ which can be used to define
for each
ladder system $\mathcal{C}$ an injective map $\pow{\omega}\to \omega_2$
with assignment $r\mapsto \alpha$ of the real $r$ to the ordinal $\alpha$ least such that
$\psi(\mathcal{C},r,\alpha)$.

\item The signature $\tau_{\NS_{\omega_1}}\cup\UB$ 
makes the $\Pi_1$-theory of $V$ invariant across the generic multiverse (cfr. Thm. \ref{Thm:mainthm-8}); 
hence we can use forcing
to detect which $\Pi_2$-sentences should belong to the model companion of set theory in any signature
$\tau\subseteq \tau_{\NS_{\omega_1}}\cup\UB$ (if such a model companion exists); this is exactly the argument we used to argue for $\neg\mathsf{CH}$.

\item $\stUB$ with a weak form of sharp for universally Baire sets can be equivalently formulated as the assertion that the $\tau_{\NS_{\omega_1}}\cup\UB$-theory of $V$ has as model companion the 
$\tau_{\NS_{\omega_1}}\cup\UB$-theory of $H_{\omega_2}$ (cfr. Thm. \ref{Thm:mainthm-7}). This brings to light the complete accordance between the philosophy driving $\Pmax$ and bounded forcing axioms (that of maximizing the $\Pi_2$-sentences true in $H_{\omega_2}$) with the notion of model companionship.
\end{itemize}
%

 

\subsubsection*{Model companionship in set theory}
Model companionship is a tameness notion which must 
be handled with care 
(see Section \ref{subsec:tameness-modcompan}).
We believe that the present paper presents a 
reasonable test to gauge the tameness of this notion:
in set theory we are focusing mostly in two types of structures:
 generic extensions $V[G]$ of the universe of sets $V$ produced 
by (certain types of) forcings $P$, and the theory of $H_\lambda^{V[G]}$ of these 
generic extensions for suitably chosen
(and definable) cardinals $\lambda$. 
We often study
these structures working in signatures $\tau$ maintaining that
$V\sqsubseteq V[G]$ and 
$H_{\lambda}^{V[G]}\prec_1 V[G]$ also for $\tau$
(in particular here and in a huge number of works one consider the case of 
$\tau$ being $\tau_{\NS_{\omega_1}}$, $P$ being a stationary set preserving forcing, 
$\lambda$ being $\omega_2$, or the case $\tau$ being $\sigma_{\omega}$, 
$P$ being any forcing, $\lambda$ being $\omega_1$).
The results of the present paper (and of its sequel) show that the axiomatization of \emph{set theory+large cardinals}
in these signatures is well behaved: first of all the models of its $\Pi_1$-fragment include all the structures of interest, i.e. all generic extensions of $V$ (eventually obtained by forcing of a certain kind), and all the initial segments of these generic extensions containing a large enough chunk of the universe. 
Moreover this theory admits a model companion and this model companion is uniquely determined by the family of 
$\Pi_2$-sentences which we can provably force to hold in the appropriate 
$H_\lambda^{V[G]}$ (with $\lambda=\omega_1$ or $\lambda=\omega_2$ decided by the signature).
It has also to be noted that even the substructure relation is not that much affected by forcing; for example  any $G$ 
$V$-generic for a stationary set preserving forcing $P$
maintains that $V\sqsubseteq V[G]$ also for the signature
$\sigma_{\omega,\NS_{\omega_1}}$. (May be surprisingly) Thm. \ref{Thm:mainthm-8}
shows that if $P$ is not stationary set preserving 
$V\sqsubseteq V[G]$ fails for 
$\sigma_{\omega,\NS_{\omega_1}}$, nonetheless $V$ and $V[G]$ will satisfy exactly the same $\Pi_1$-sentences for $\sigma_{\omega,\NS_{\omega_1}}$.

It is in our eyes surprising the perfect matching existing
between generic absoluteness results and the notion of model companionship
which the present paper reveals. 

\section{The theory of $H_{\kappa^+}$ is the model companion of set theory}\label{sec:Hkappa+}

\begin{notation}\label{not:modthnot2}
Given a $\in$-structure $(M,E)$ and $\tau$ a signature
among $\tau_\ST,\sigma_\kappa,\dots$,
from now we let
$(M,\tau^M)$ be the unique extension of 
$(M,E)$ defined in accordance with 
Notation \ref{not:keynotation} and Fact \ref{fac:basicfact}.
In particular $(M,\tau^M)$ is a shorthand for 
$(M,S^M:S\in\tau)$.
If $(N,E)$ is a substructure of $(M,E)$ we also write
$(N,\tau^M)$ as a shorthand for 
$(N,S^M\restriction N:S\in\tau)$.
\end{notation}

\subsection{By-interpretability of the first order theory of $H_{\kappa^+}$
with the first order theory of $\pow{\kappa}$}
\label{subsec:secordequiv}

Let's compare the first order theory of the structure
\[
(\pow{\kappa},S_\phi^V:\phi\in\tau_{\ST})
\]
with 
that of the $\tau_\ST$-theory of $H_{\kappa^+}$ in models of $\ZFC_{\ST}$. 
We will show that they
are $\ZFC_{\tau_{\ST}}$-provably by-interpretable with a by-interpetation translating $H_{\kappa^+}$ in a $\Pi_1$-definable subset of  $\pow{\kappa^2}$ and
atomic predicates into 
$\Sigma_1$-relations over this set. 
This result is the key to the proof of 
Thm. \ref{Thm:mainthm-1} and it is just outlining the model theoretic consequences of the well-known fact that sets can be coded by well-founded extensional graphs.

\begin{definition}
Given $a\in H_{\kappa^+}$, $R\in \pow{\kappa^2}$ codes $a$, if
$R$ codes a well-founded extensional relation on 
some $\alpha\leq\kappa$ with top element $0$
so that the transitive collapse mapping of $(\alpha,R)$ maps $0$ to $a$.

\begin{itemize}
\item
$\WFE_\kappa$ is the set of $R\in \pow{\kappa}$ which 
are a well founded extensional relation with domain 
$\alpha\leq\kappa$ and top element $0$.
\item
 $\Cod_\kappa:\WFE_\kappa\to H_{\kappa^+}$ is the map assigning $a$ to $R$ if and only if 
 $R$ codes $a$.
\end{itemize}
\end{definition}

The following theorem shows that the structure $(H_{\kappa^+},\in)$ is interpreted by means of ``imaginaries'' in the structure
$(\pow{\kappa},\tau_{\ST}^V)$ by means of:
\begin{itemize}
    \item a universal $\tau_\ST\cup\bp{\kappa}$-formula (with quantifiers
    ranging over subsets of $\kappa^{<\omega}$)
    defining a set $\WFE_\kappa\subseteq\pow{\kappa^2}$.
    \item an equivalence relation $\cong_\kappa$ on $\WFE_\kappa$
    defined by an existential $\tau_\ST\cup\bp{\kappa}$-formula (with quantifiers
    ranging over subsets of $\kappa^{<\omega}$)
    \item A binary relation $E_\kappa$ on $\WFE_\kappa$
    invariant under $\cong_\kappa$ representing the $\in$-relation as the extension of 
an existential $\tau_\ST\cup\bp{\kappa}$-formula (with quantifiers
    ranging over subsets of $\kappa^{<\omega}$)\footnote{See \cite[Section 25]{JECHST} for proofs of the case $\kappa=\omega$; in particular the statement and proof of Lemma 25.25 and the proof of \cite[Thm. 13.28]{JECHST} contain all ideas on which one can elaborate to draw the conclusions of Thm.~\ref{thm:keypropCod}.}.
\end{itemize}
\begin{theorem}\label{thm:keypropCod}
Assume $\ZFC^{-}_{\kappa}$. The following holds\footnote{Many transitive supersets of $H_{\kappa^+}$ are 
$\tau_\ST\cup\bp{\kappa}$-model of $\ZFC^{-}_{\kappa}$ for $\kappa$ an infinite cardinal (see \cite[Section IV.6]{KUNEN}). 
To simplify notation we assume to have fixed a transitive 
$\tau_\ST\cup\bp{\kappa}$-model $\mathcal{N}$ of $\ZFC^{-}_\kappa$
with domain $N\supseteq H_{\kappa^+}$. The reader can easily realize that all these statements holds for an arbitrary model $\mathcal{N}$ of $\ZFC^-_\kappa$ replacing $H_{\kappa^+}$ with its version according to $\mathcal{N}$.}:
 \begin{enumerate}
\item
The map $\mathrm{Cod}_\kappa$ and $\WFE_\kappa$ are defined by $\ZFC^-_\kappa$-provably 
$\Delta_1$-properties  in parameter $\kappa$. Moreover $\Cod_\kappa:\WFE_\kappa\to H_{\kappa^+}$
is surjective (provably in $\ZFC^{-}_{\kappa}$), and
$\WFE_\kappa$ is defined by a universal 
$\tau_\ST\cup\bp{\kappa}$-formula with quantifiers
ranging over subsets of $\kappa^{<\omega}$.
\item 
There are existential $\tau_\ST\cup\bp{\kappa}$-formulae (with quantifiers
    ranging over subsets of $\kappa^{<\omega}$), $\phi_\in,\phi_=$ such that
for all $R,S\in \WFE_\kappa$, $\phi_=(R,S)$ if and only if $\Cod_\kappa(R)=\Cod_\kappa(S)$ and 
$\phi_\in(R,S)$  if and only if $\Cod_\kappa(R)\in\Cod_\kappa(S)$. In particular letting
\[
E_\kappa=\bp{(R,S)\in \WFE_\kappa: \phi_\in(R,S)},
\]
\[
\cong_\kappa=\bp{(R,S)\in \WFE_\kappa: \phi_=(R,S)},
\]
$\cong_\kappa$ is a $\ZFC^-_\kappa$-provably definable equivalence relation, $E_\kappa$ respects it, and
\[
(\WFE_\kappa/_{\cong_\kappa}, E_\kappa/_{\cong_\kappa})
\]
is 
isomorphic to $(H_{\kappa^+},\in)$ via the map $[R]\mapsto \Cod_\kappa(R)$.
\end{enumerate}
\end{theorem}

\begin{proof}
A detailed proof requires a careful examination of the syntactic properties of $\Delta_0$-formulae, in line with the one carried in Kunen's \cite[Chapter IV]{KUNEN}.
We outline the main ideas, 
following Kunen's book terminology for certain 
set theoretic operations on sets, functions and relations (such as $\dom(f),\ran(f)$, $\text{Ext}(R)$, etc). 
To simplify the notation, we prove the results
for a transitive model $(N,\in)$ which is then extended
to a structure $(N,\tau_\ST^N,\kappa^N)$ which 
models $\ZFC^-_\kappa$, and whose domain contains $H_{\kappa^+}$. 
The reader can verify by itself that the argument 
is modular and works for any other model of 
$\ZFC^-_\kappa$ 
(transitive or ill-founded, containing the ``true'' $H_{\kappa^+}$ or not). 
\begin{enumerate} 
\item This is proved in details in \cite[Chapter IV]{KUNEN}.
To define $\WFE_\kappa$ by a universal property over subsets of $\kappa$
and $\Cod_\kappa$ by a $\Delta_1$-property over $H_{\kappa^+}$, we proceed
as follows:
\begin{itemize}
\item
\emph{$R$ is an extensional relation with domain contained in 
$\kappa$ and top element $0$}
is defined by the $\tau_\ST\cup\bp{\kappa}$-atomic formula 
$\psi_{\mathrm{EXT}}(R)$ $\ZFC^{-}_\kappa$-provably equivalent to the $\Delta_0(\kappa)$-formula:
\begin{align*}
(R\subseteq\kappa^2)\wedge \\
\wedge (\text{Ext}(R)\in\kappa\vee \text{Ext}(R)=\kappa)
\wedge\\
\wedge\forall \alpha,\beta\in\text{Ext}(R)\,
[\forall u\in\text{Ext}(R)\,(u\mathrel{R}\alpha\leftrightarrow u\mathrel{R}\beta)\rightarrow (\alpha=\beta)]\wedge\\
\wedge \forall \alpha\in\text{Ext}(R)\,\neg (0\mathrel{R}\alpha).
\end{align*}
\item 
$\WFE_\kappa$ is defined by the universal $\tau_\ST\cup\bp{\kappa}$-formula 
$\phi_{\WFE_\kappa}(R)$ (quantifying only over subsets of $\kappa^{<\omega}$)
\begin{align*}
\psi_{\mathrm{EXT}}(R)\wedge \\
\wedge [\forall f\subseteq \kappa^2\,(f\text{ is a function }\rightarrow\exists n\in\omega \,
\neg(\ap{f(n+1),f(n)}\in R))].
\end{align*}
Its interpretation is the subset of $\pow{\kappa^{<\omega}}$ of the $\sigma_\kappa$-symbol 
$S_{\phi_{\WFE_\kappa}}$. 

\item To define $\Cod_\kappa$,
consider the $\tau_\ST\cup\bp{\kappa}$-atomic formula 
$\psi_{\Cod}(G,R)$ provably equivalent to the $\tau_\ST\cup\bp{\kappa}$-formula:
\begin{align*}
\psi_{\mathrm{EXT}}(R)\wedge\\
\wedge (G\text{ is a function})\wedge\\
\wedge (\dom(G)=\text{Ext}(R))\wedge\\ \wedge\forall\alpha,\beta\in\text{Ext}(R)\,[\alpha\mathrel{R}\beta\leftrightarrow G(\alpha)\in G(\beta)].
\end{align*}

Then $\Cod_\kappa(R)=a$ can be defined either by the existential $\tau_\ST\cup\bp{\kappa}$-formula\footnote{Given an $R$ such that $\psi_{\mathrm{EXT}}(R)$ holds,
\emph{$R$ is a well founded relation} holds in a model of 
$\ZFC^-_\kappa$
if and only if 
$\Cod_\kappa$ is defined on $R$. In the theory $\ZFC^-_\kappa$, $\WFE_\kappa$ can be defined using a universal property 
by a $\tau_\ST\cup\bp{\kappa}$-formula quantifying only over subsets of $\kappa$. On the other hand if we allow arbitrary quantification over elements 
of $H_{\kappa^+}$, we can express the well-foundedness of $R$ also using the existential formula 
$\exists G\,\psi_{\Cod_\kappa}(G,R)$. This is why $\WFE_\kappa$ is defined by a universal 
$\tau_\ST\cup\bp{\kappa}$-property in the structure $(\pow{\kappa},\tau_\ST^V,\kappa)$,
while the graph of $\Cod_\kappa$ can be defined by a $\Delta_1$-property for $\tau_\ST\cup\bp{\kappa}$
in the structure $(H_{\kappa^+},\tau_\ST^V,\kappa^V)$.} 
\[
\exists G\,(\psi_{\Cod}(G,R)\wedge G(0)=a)
\]
or by the universal $\tau_\ST\cup\bp{\kappa}$-formula 
\[
\forall G\,(\psi_{\Cod}(G,R)\rightarrow G(0)=a).
\]
\end{itemize}

\item The equality relation in $H_{\kappa^+}$ is transferred to the isomorphism relation
between elements of $\WFE_\kappa$: if $R,S$ are well-founded extensional on $\kappa$ with a top-element,
the Mostowski collapsing theorem entails that $\Cod_\kappa(R)=\Cod_\kappa(S)$ if and only if 
$(\mathrm{Ext}(R),R)\cong(\mathrm{Ext}(S),S)$. 
Isomorphism of the two structures $(\mathrm{Ext}(R),R)\cong(\mathrm{Ext}(S),S)$ is expressed by the $\Sigma_1$-formula
for $\tau_{\kappa}$:
\[
\phi_=(R,S)\equiv 
\exists f\,(f \text{ is a bijection of $\kappa$ onto $\kappa$ and $\alpha R\beta$ if and only if 
$f(\alpha) S f(\beta)$}).
\]
In particular we get that $S_{\phi_=}(R,S)$ holds in $H_{\kappa^+}$ for $R,S\in \WFE_\kappa$
if and only if $\Cod_\kappa(R)=\Cod_\kappa(S)$.

Similarly one can express $\Cod_\kappa(R)\in\Cod_\kappa(S)$ by the $\Sigma_1$-property $\phi_\in$
in $\tau_{\kappa}$
stating that
$(\mathrm{Ext}(R),R)$ is isomorphic to $(\mathrm{pred}_S(\alpha),S)$ for some $\alpha\in\kappa$ with $\alpha \mathrel{S}0$, 
where $\mathrm{pred}_S(\alpha)$ is given by the elements
of $\mathrm{Ext}(S)$ which are connected by a finite path to $\alpha$. 

Moreover letting $\cong_\kappa\subseteq \WFE_\kappa^2$ denote the isomorphism relation between elements of $\WFE_\kappa$
and $E_\kappa\subseteq \WFE_\kappa^2$ denote the relation which translates into the $\in$-relation via $\Cod_\kappa$, 
it is clear that $\cong_\kappa$ is a congruence relation
over $E_\kappa$, i.e.: if $R_0 \cong_\kappa R_1$ and $S_0\cong_\kappa S_1$,
$R_0 \mathrel{E_\kappa} S_0$ if and only if $R_1 \mathrel{E_\kappa}  S_1$.

This gives that the structure $(\WFE_\kappa/_{\cong_\kappa}, E_\kappa/_{\cong_\kappa})$ is 
isomorphic to $(H_{\kappa^+},\in)$ via the map $[R]\mapsto \Cod_\kappa(R)$ 
(where $\WFE_\kappa/_{\cong_\kappa}$ is the set
of equivalence classes of $\cong_\kappa$ and the quotient relation $[R] \mathrel{E_\kappa/_{\cong_\kappa}} [S]$ holds 
if and only if $R \mathrel{E_\kappa}  S$).

This isomorphism is defined via the map $\Cod_\kappa$, which is by itself defined by 
a $\ZFC^-_\kappa$-provably $\Delta_1$-property for $\tau_\ST\cup\bp{\kappa}$.

The very definition of $\WFE_\kappa,\cong_\kappa,E_\kappa$ show that
\[
\WFE_\kappa=S_{\phi_{\WFE_\kappa}}^{N},
\]
\[
\cong_\kappa=S_{\phi_{\WFE_\kappa}(x)\wedge \phi_{\WFE_\kappa}(y)\wedge \phi_{=}(x,y)}^{N},
\]
\[
E_\kappa=S_{\phi_{\WFE_\kappa}(x)\wedge \phi_{\WFE_\kappa}(y)\wedge \phi_{\in}(x,y)}^{N}.
\]
\end{enumerate}
\end{proof}

\subsection{Model completeness for the theory of $H_{\kappa^+}$}

\begin{theorem}\label{thm:modcompHkappa+}
Any $\sigma_\kappa$-theory $T$ extending 
\[
\ZFC^{*-}_\kappa\cup\bp{\text{all sets have size $\kappa$}}
\]
is model complete.
\end{theorem}

\begin{proof}
To simplify notation, 
we conform to the assumption of the previous theorem, 
i.e. we assume that the model $(N,\in)$ which is uniquely extended to a model of 
$\ZFC^{*-}_\kappa+$\emph{ every set has size $\kappa$}
on which we work is a transitive superstructure of 
$H_{\kappa^+}$.

The statement \emph{every set has size $\kappa$} is satisified by a
$\ZFC^-_\kappa$-model $(N,\tau_\ST^V,\kappa)$ with 
$N\supseteq H_\kappa^+$ if and only if $N=H_{\kappa^+}$.
From now on we proceed assuming this equality.

By Robinson's test \ref{lem:robtest} it suffices to show that
for all $\in$-formulae $\phi(\vec{x})$ 
\[
\ZFC^-_\kappa+\text{ every set has size $\kappa$}\vdash
\forall\vec{x} \,(\phi(\vec{x})\leftrightarrow\psi_\phi(\vec{x})),
\]
for some universal $\sigma_\kappa$-formula $\psi_\phi$.

We will first define a recursive map 
$\phi\to\theta_\phi$ which maps $\Sigma_n$-formulae $\phi$ for $\bp{\in,\kappa}$ 
quantifying over all elements of
$H_{\kappa^+}$ to $\Sigma_{n+1}$-formulae $\theta_\phi$ for $\tau_\ST\cup\bp{\kappa}$
whose quantifier range just over subsets of $\kappa^{<\omega}$.

The proof of the previous theorem gave $\tau_\ST\cup\bp{\kappa}$-formulae
$\theta_{x=y}$, $\theta_{x\in y}$ such that
\[
S_{\theta_{x=y}}^{H_{\kappa^+}}=\cong_\kappa=\bp{(R,S)\in (\WFE_\kappa)^2:\, \Cod_\kappa(R)=\Cod_\kappa(S)},
\]
\[
S_{\theta_{x\in y}}^{H_{\kappa^+}}=E_\kappa=\bp{(R,S)\in (\WFE_\kappa)^2:\, \Cod_\kappa(R)\in\Cod_\kappa(S)}.
\]
Specifically (following the notation of that proof)
\[
\theta_{x=y}=\phi_{\WFE_\kappa}(x)\wedge \phi_{\WFE_\kappa}(y)\wedge \phi_{=}(x,y),
\]
\[
\theta_{x\in y}=\phi_{\WFE_\kappa}(x)\wedge \phi_{\WFE_\kappa}(y)\wedge \phi_{\in}(x,y).
\]

Now for any $\bp{\in,\kappa}$-formula $\psi(\vec{x})$, 
we proceed to define the $\tau_\ST\cup\bp{\kappa}$-formula $\theta_\psi(\vec{x})$ letting:
\begin{itemize}
\item $\theta_{\psi\wedge\psi}(\vec{x})$ be $\theta_{\psi}(\vec{x})\wedge\theta_{\psi}(\vec{x})$,
\item $\theta_{\neg\psi}(\vec{x})$ be $\neg\theta_{\psi}(\vec{x})$,
\item $\theta_{\exists y\psi(y,\vec{x})}(\vec{x})$ be $\exists y\theta_{\psi}(y,\vec{x})\wedge \phi_{\WFE_\kappa}(y)$.
\end{itemize}
An easy induction on the complexity of the $\tau_\ST\cup\bp{\kappa}$-formulae $\theta_\phi(\vec{x})$
gives that for any $\bp{\in,\kappa}$-definable subset
$A$ of $(H_{\kappa^+})^n$ which is the extension of some $\bp{\in,\kappa}$-formula $\phi(x_1,\dots,x_n)$ 
\[
\bp{(R_1,\dots,R_n)\in (\WFE_\kappa)^n:\, (\Cod_\kappa(R_1),\dots,\Cod_\kappa(R_n))\in A}=
S_{\theta_\phi}^{H_{\kappa^+}},
\] 
with the further property that $S_{\theta_\phi}^{H_{\kappa^+}}\subseteq (\WFE_{\kappa})^n$ respects the
$\cong_\kappa$-relation\footnote{It is also clear from our argument that the map $\phi\mapsto\theta_\phi$ is recursive (and a careful inspection
reveals that it maps a $\Sigma_n$-formula
to a $\Sigma_{n+1}$-formula).}.

Now every $\sigma_\kappa$-formula is $\ZFC^{*-}_\kappa$-equivalent to a 
$\bp{\in,\kappa}$-formula\footnote{The map assigning to any $\sigma_\kappa$-formula a 
$\ZFC^{*-}_\kappa$-equivalent
$\bp{\in,\kappa}$-formula can also be chosen to be recursive.}.

Therefore we can extend $\phi\mapsto\theta_\phi$ 
assigning to any $\sigma_\kappa$-formula $\phi(\vec{x})$ the formula $\theta_\psi(\vec{x})$ for some 
$\bp{\in,\kappa}$-formula $\psi(\vec{x})$ which is $\ZFC^{*-}_\kappa$-equivalent to $\phi(\vec{x})$.

Then for any $\bp{\in,\kappa}$-formula $\phi(x_1,\dots,x_n)$
$H_{\kappa^+}\models \phi(a_1,\dots,a_n)$ if and only if 
\[
(\WFE_\kappa/_{\cong_\kappa}, E_\kappa/_{\cong_\kappa})\models \phi([R_1],\dots,[R_n])
\]
with $\Cod_\kappa(R_i)=a_i$ for $i=1,\dots,n$
if and only if 
\[
H_{\kappa^+}\models \forall R_1,\dots,R_n\, [(\bigwedge_{i=1}^n \Cod_\kappa(R_i)=a_i)\rightarrow\theta_\phi(R_1,\dots,R_n)]
\]
if and only if 
\[
H_{\kappa^+}\models \forall R_1,\dots,R_n\, [(\bigwedge_{i=1}^n \Cod_\kappa(R_i)=a_i)\rightarrow S_{\theta_\phi}(R_1,\dots,R_n)].
\]

Since this argument can be repeated verbatim for any model of 
$\ZFC^{*-}_\kappa$+\emph{every set has size $\kappa$}, and any $\sigma_\kappa$-formula
 is $\ZFC^{*-}_\kappa$-equivalent to a $\bp{\in,\kappa}$-formula,
we have proved the following:
\begin{claim}
For any $\sigma_\kappa$-formula $\phi(x_1,\dots,x_n)$,
$\ZFC^{*-}_\kappa$+\emph{every set has size $\kappa$} proves that
\[
\forall x_1,\dots,x_n\,[\phi(x_1,\dots,x_n)\leftrightarrow \forall y_1,\dots,y_n\,[(\bigwedge_{i=1}^n \Cod_\kappa(y_i)=x_i)\rightarrow S_{\theta_\phi}(y_1,\dots,y_n)]].
\]
\end{claim}
But $\Cod_\kappa(y)=x$ is expressible by an existential
$\tau_\ST\cup\bp{\kappa}$-formula 
provably in $\ZFC^-_\kappa\subseteq \ZFC^{*-}_\kappa$, therefore
\[
\forall y_1,\dots,y_n\,[(\bigwedge_{i=1}^n \Cod_\kappa(y_i)=x_i)\rightarrow S_{\theta_\phi}(y_1,\dots,y_n)]
\]
is a universal $\sigma_\kappa$-formula, and we are done. 
\end{proof}

\subsection{Proof of Thm.~\ref{Thm:mainthm-1}}

We can immediately prove Thm.~\ref{Thm:mainthm-1}.

\begin{proof}
By Thm. \ref{thm:keypropCod}, any theory extending
\[
\ZFC^{*-}_\kappa+\emph{every set has size $\kappa$}
\] 
is model complete.
Therefore so is 
\[
T^*=T_\forall\cup\ZFC^{*-}_\kappa+\emph{every set has size $\kappa$}.
\] 
We need to show that $T^*$ is the model companion of $T$, and that
$T^*=T^*_i$ for $i=0,1$ where 
\[
T^*_0=\bp{\psi:\, \psi\text{ is a $\Pi_2$-sentence for $\sigma_\kappa$ and }T\vdash\psi^{H_{\kappa^+}} },
\]
and $T^*_1$ is the set of $\Pi_2$-sentences $\phi$ such that
\begin{quote}
For all $\Pi_1$-sentences $\theta$ for $\tau$
$T_\forall+\phi$ is consistent if and only if so is 
$T_\forall+\phi+\theta$.
 \end{quote}

\begin{description}
\item[$T^*$ is the model companion of $T$]
By Lemma \ref{fac:proofthm1-2}(\ref{fac:proofthm1-2-a}).
It suffices to verify that for every model 
$\mathcal{M}$ of $T$, 
 $H_{\kappa^+}^\mathcal{M}$ is a $\Sigma_1$-elementary substructure of $\mathcal{M}$
 which models $T^*$.
But this holds true by 
Lemma~\ref{lem:levyabsHkappa+}.
Therefore $T^*$ is a model companion for $T$.

\item[$T^*_1=T^*$]
By Lemma~\ref{fac:proofthm1-2}(\ref{fac:proofthm1-2-c}) 
the model companion of $T$
is axiomatized by $T^*_1$.

\item[$T^*_0=T^*$]
First assume $\psi$ is a $\Pi_2$-sentence in $T^*$ and 
$\mathcal{M}$ models $T$. We must show that 
$H_{\kappa^+}^{\mathcal{M}}$ models $\psi$.
But this is the case since $H_{\kappa^+}^{\mathcal{M}}$ models 
$T^*$.

Conversely assume $\psi$ is a $\Pi_2$-sentence for $\sigma_\kappa$ which holds in any 
$H_{\kappa^+}^{\mathcal{M}}$ for $\mathcal{M}$ a model of $T$. We must show that $\psi\in T^*$.
We show that $\psi\in T^*_1$:
(using 
Lemma~\ref{fac:proofthm1-2}(\ref{fac:proofthm1-2-b}))
it suffices to show that 
$S_\forall\cup\bp{\psi}$ is consistent
for any consistent $S\supseteq T$:
fix $\mathcal{M}$ a model of $S$;
by assumption $H_{\kappa^+}^{\mathcal{M}}$ models $\psi$; by 
Lemma~\ref{lem:levyabsHkappa+} applied to $\mathcal{M}$, we get that
$H_{\kappa^+}^{\mathcal{M}}$ models $S_\forall$; we conclude that
$S_\forall\cup\bp{\psi}$ is consistent.
\end{description}
The proof is completed.
\end{proof}

\begin{remark}
Thm. \ref{Thm:mainthm-1} can be proved for 
many other signatures other than $\sigma_\kappa$.
It suffices that the signature in question adds new predicates just for definable subsets of $\pow{\kappa}^n$,
and also that it adds family of predicates which are closed under definability (i.e. projections, complementation, finite unions, permutations) and under the map $\Cod_\kappa$.
Under these assumptions we can still use 
Lemma \ref{lem:levyabsHkappa+} and 
Lemma \ref{fac:proofthm1-2} to argue for the evident declination of Thm. \ref{Thm:mainthm-1} to this set up.
However linking it to generic absoluteness results as we did in Theorem \ref{Thm:mainthm-3}
requires much more care in the definition of the signature.
We will pursue this matter in more details in the next section and in a follow-up of this paper.  
\end{remark}

\subsection{A weak version of Theorem \ref{Thm:mainthm-4}}

Let $\ZFC^*_{\omega_1}\supseteq\ZFC_\ST$ be the 
$\sigma_{\omega_1}=\sigma_\omega\cup\bp{\kappa}$-theory
obtained adding axioms which force in each of its 
$\sigma_{\omega_1}$-models
$\kappa$ to be interpreted by the first uncountable cardinal, and
each predicate symbol $S_\phi$ to be interpreted as the subset of 
$\pow{\omega_1^{<\omega}}^n$ defined by $\phi^{\pow{\omega_1^{<\omega}}}(x_1,\dots,x_n)$
(see again Notation \ref{not:keynotation} and Fact \ref{fac:basicfact} for details).


\begin{Theorem}\label{Thm:mainthm-5}
Let
$T$ be a $\sigma_{\omega_1}$-theory extending  $\ZFC^*_{\omega_1}$ 
with the $\in$-sentence:
\[
\emph{There are class many superhuge cardinals},
\]
and such that $T+\MM^{+++}$ is consistent.

TFAE for any
$\Pi_2$-sentence $\psi$ for $\sigma_{\omega_1}$:
\begin{enumerate}
\item \label{Thm:mainthm-5-2}
For all universal $\sigma_{\omega_1}$-sentences $\theta$
such that $T+\theta$ is consistent, so is $T_\forall+\theta+\psi$;
\item \label{Thm:mainthm-5-3}
$T+\MM^{+++}$ proves that some stationary set preserving forcing notion $P$ forces 
$\psi^{\dot{H}_{\omega_2}}+\MM^{+++}$;
\item \label{Thm:mainthm-5-1}
$T+\MM^{+++}\vdash \psi^{H_{\omega_2}}$.
\end{enumerate}
\end{Theorem}
See Remarks \ref{Rem:keyrem1}(\ref{Rem:keyrem1-4}) for some information on $\MM^{+++}$,
and \ref{Rem:keyrem1}(\ref{Rem:keyrem1-3}) for informations on superhugeness.

The proof of Theorem \ref{Thm:mainthm-5} is a trivial variation of the proof of Theorem \ref{Thm:mainthm-3}:
\begin{proof}
\cite[Thm. 5.18]{VIAMM+++} gives that
\ref{Thm:mainthm-5}(\ref{Thm:mainthm-5-1}) and \ref{Thm:mainthm-5}(\ref{Thm:mainthm-5-3}) 
are equivalent.
Theorem \ref{Thm:mainthm-1}  establishes the equivalence of \ref{Thm:mainthm-5}(\ref{Thm:mainthm-5-1}) and
\ref{Thm:mainthm-5}(\ref{Thm:mainthm-5-2}). 
\end{proof}

\begin{remark}\label{Rem:keyrem1}
\emph{}

\begin{enumerate}
\item\label{Rem:keyrem1-1}
Note that
$\ZFC^*_{\omega_1}$ is more expressive than $\ZFC^*_{\omega,\NS_{\omega_1}}$.
The former adds predicate symbols for all subsets of $\pow{\omega_1^{<\omega}}^k$ 
defined by $\phi^{\pow{\omega_1^{<\omega}}}(x_1,\dots,x_k)$ as $\phi$ ranges over the 
$\in$-formulae.
The latter adds predicate symbols for all subsets of $\pow{\omega^{<\omega}}^k$ defined by 
$\phi^{\pow{\omega^{<\omega}}}(x_1,\dots,x_k)$ as $\phi$ ranges over the $\in$-formulae and a 
unique predicate symbol for the
subset of $\pow{\omega_1}$ given by the non-stationary ideal.

More precisely for any model $\mathcal{M}=(M,E)$ of $\ZFC$, if $\mathcal{M}_0$
is the unique extension of $\mathcal{M}$ to a $\sigma_{\omega_1}$-model
of $\ZFC^*_{\omega_1}$, and $\mathcal{M}_1$
is the unique extension of $\mathcal{M}$ to a $\sigma_{\omega,\NS_{\omega_1}}$-model
of $\ZFC^*_{\omega,\NS_{\omega_1}}$, 
we get that $R_\psi^{\mathcal{M}_0}=R_\psi^{\mathcal{M}_1}$ 
and $f_\psi^{\mathcal{M}_0}=f_\psi^{\mathcal{M}_1}$ for all bounded formulae $\psi$,
$\omega^{\mathcal{M}_0}=\omega^{\mathcal{M}_1}$, 
$\omega_1^{\mathcal{M}_0}=\omega_1^{\mathcal{M}_1}$, but for any $\tau_\ST$-formula $\phi$,
$S_\phi^{\mathcal{M}_1}=S_{\phi^{\pow{\omega^{<\omega}}}}^{\mathcal{M}_0}$.

\item\label{Rem:keyrem1-2}
A key distinction between the signatures 
$\sigma_{\omega_1}$ and 
$\sigma_{\omega,\NS_{\omega_1}}$ is that (assuming large cardinals) $\mathsf{CH}$ cannot be 
$T$-equivalent to a
$\Sigma_1$-sentence\footnote{By Thm. \ref{Thm:mainthm-8}.} in $\sigma_{\omega,\NS_{\omega_1}}$ for any $T$ as in the assumptions of Thm. \ref{Thm:mainthm-4}, 
while it is $\ZFC_{\omega_1}$-equivalent to
an atomic $\sigma_{\omega_1}$-sentence\footnote{Following the notation to be introduced in 
Section~\ref{sec:Hkappa+}, $\mathsf{CH}$ can be expressed as the $\tau_\ST\cup\bp{\omega_1}$-sentence
quantifying just over subsets of $\pow{\omega_1^{<\omega}}$:
\[
\exists R\subseteq\pow{\omega_1^{<\omega}}[\WFE_{\omega_1}(R)\wedge\forall S \subseteq\pow{\omega_1^{<\omega}}
[(\WFE_{\omega_1}(S)\wedge\mathrm{Ext}(S)=\omega)\leftrightarrow S \mathrel{E_{\omega_1}} R]].
\]
The latter is equivalent to a $\sigma_{\omega_1}$-sentence in $\ZFC^*_{\omega_1}$.}.
$\neg\mathsf{CH}$  is the simplest example of the 
type of $\Pi_2$-sentences which exemplifies why
 Thm. \ref{Thm:mainthm-5}(\ref{Thm:mainthm-5-3}) must be weakened with respect to
 Thm. \ref{Thm:mainthm-4}(\ref{Thm:mainthm-4-3}) and why Thm. \ref{Thm:mainthm-4} needs a different proof strategy than the one we use here to 
establish Theorems \ref{Thm:mainthm-3} and \ref{Thm:mainthm-5}  (see for details \ref{Rem:keyrem1-6} below).
On the other hand the family of $\Pi_2$-sentences $\psi$ 
to which Theorem \ref{Thm:mainthm-5} applies
is larger than the ones considered in Theorem \ref{Thm:mainthm-4} because the signature $\sigma_{\omega_1}$ is more expressive than 
$\sigma_{\omega,\NS_{\omega_1}}$
(as shown by the case for $\mathsf{CH}$).

\item\label{Rem:keyrem1-3}
$\delta$ is superhuge if it supercompact and this can be witnessed by huge embeddings.
A superhuge cardinal is consistent
relative to the existence of a $2$-huge cardinal.

\item\label{Rem:keyrem1-4}
For a definition of $\MM^{+++}$ see \cite[Def. 5.19]{VIAMM+++}. We just note that
$\MM^{+++}$ is a natural strengthening of $(*)$-$\UB$ (by the recent breakthrough of 
Asper\`o and Schindler \cite{ASPSCH(*)}) and of Martin's maximum 
(for example any
of the standard iterations to produce a model of Martin's maximum produce a model of 
$\MM^{+++}$
if the iteration has length a superhuge cardinal \cite[Thm 5.29]{VIAMM+++}).

\item\label{Rem:keyrem1-5}
We can prove exactly the same results of Thm. \ref{Thm:mainthm-5} replacing 
(verbatim in its statement)
$\MM^{+++}$ by any of the axioms $\mathsf{RA}_\omega(\Gamma)$ introduced in \cite{VIAAUD14} or the axioms $\mathsf{CFA}(\Gamma)$ introduced in \cite{VIAASP};
in item \ref{Thm:mainthm-5}(\ref{Thm:mainthm-5-3})
\emph{stationary set preserving forcing notion $P$}
must be replaced by $P\in\Gamma$.
\item We consider Thm. \ref{Thm:mainthm-5} weaker than Thm. \ref{Thm:mainthm-4}, because
in Thm. \ref{Thm:mainthm-4} one can choose the theory $T$ to be inconsistent with 
$\maxUB+(*)$-$\UB$ without hampering its conclusion (for example $T$ could satisfy $\mathsf{CH}$,
a statement denied by $(*)$-$\UB$),
and because \ref{Thm:mainthm-4}(\ref{Thm:mainthm-4-3}) holds for all forcing notions $P$.
The key point separating these two results is that the signature $\sigma_{\omega_1}$ is too 
expressive and renders many statements incompatible with forcing axioms formalizable by existential 
(or even atomic)
$\sigma_{\omega_1}$-sentences
(for example such is the case for $\mathsf{CH}$). 

\item\label{Rem:keyrem1-6}
We can also give a detailed explanation of why we cannot use 
Thm. \ref{Thm:mainthm-1}
to prove Thm. \ref{Thm:mainthm-4} as we did for Theorems \ref{Thm:mainthm-3} and \ref{Thm:mainthm-5}.
The key point is that the model companion $T^*$ of some 
$T\supseteq \ZFC_{\omega_1,\NS_{\omega_1}}+$\emph{there are class many Woodin}
 may not be axiomatized by the set $T^{**}$ of $\Pi_2$-sentences $\psi$ for 
$\sigma_{\omega,\NS_{\omega_1}}$
such that
$T\vdash\psi^{H_{\omega_2}}$, and 
this is what we used in the proofs of Theorems \ref{Thm:mainthm-3}, \ref{Thm:mainthm-5}.

For example this is the case for the theory 
$T=\ZFC_{\omega_1,\NS_{\omega_1}}+\mathsf{CH}+$\emph{there are class many Woodin}:
By Remark \ref{Rem:keyrem}(\ref{Rem:keyrem-5})
$\mathsf{CH}$ is expressible by the $\Sigma_2$-sentence in $\tau_{\omega_1}\cup\bp{\NS_{\omega_1}}$ $\psi_{\mathsf{CH}}$, which shows that 
(in view of Levy Absoluteness) $\mathsf{CH}$ and $\mathsf{CH}^{H_{\omega_2}}$ are $T$-equivalent.
Now
$\neg\mathsf{CH}$ is in the Kaiser hull of $T$ (which is a subset of $T^*$) being a $\Pi_2$-sentence compatible with 
$S_\forall$ for any complete $S\supseteq T$ in view of Thm. \ref{Thm:mainthm-8} and
Fact \ref{fac:charkaihullnonpi1comp}.
\end{enumerate}
\end{remark}

\section{Existentially closed structures, model completeness, model companionship}\label{sec:modth}

We present this topic expanding on \cite[Sections 3.1-3.2]{TENZIE}.
We decided to include detailed proofs since their presentation is (in some occasions) rather 
sketchy, and their focus is not exactly ours.

The first objective is to isolate necessary and sufficient conditions granting that some $\tau$-structure 
$\mathcal{M}$ embeds into some model of some $\tau$-theory\footnote{In what 
follows we conform to Notation \ref{not:modthnot2} and feel free to confuse a $\tau$-structure
$\mathcal{M}=(M,\tau^M)$ with its domain $M$
and an ordered tuple $\vec{a}\in\mathcal{M}^{<\omega}$
with its set of elements. Moreover we often
write $\mathcal{M}\models\phi(\vec{a})$ rather than
$\mathcal{M}\models\phi(\vec{x})[\vec{x}/\vec{a}]$ when
$\mathcal{M}$ is $\tau$-structure $\vec{a}\in\mathcal{M}^{<\omega}$, $\phi$ is a $\tau$-formula.} $T$.

\begin{definition}
Given $\tau$-theories $T,S$,
a $\tau$-sentence $\psi$ separates $T$ from $S$ if $T\vdash\psi$ and $S\vdash\neg\psi$.

$T$ is $\Pi_n$-separated from $S$ if some $\Pi_n$-sentence for $\tau$ 
separates $T$ from $S$.
\end{definition}

\begin{lemma}\label{lem:biembequivequnivth}
Assume $S,T$ are $\tau$-theories. 
TFAE:
\begin{enumerate}
\item \label{lem:biembequivequnivth-1}  $T$ is not $\Pi_1$-separated from $S$
(i.e. no universal sentence $\psi$ is such that $T\vdash \psi$ and $S\vdash\neg\psi$).
\item \label{lem:biembequivequnivth-3} 
There is \emph{some} $\tau$-model $\mathcal{M}$ of $S$ 
which can be embedded in 
some
$\tau$-model $\mathcal{N}$ of $T$.
\end{enumerate}
\end{lemma}

See also \cite[Lemma 3.1.1, Lemma 3.1.2, Thm. 3.1.3]{TENZIE}
\begin{proof}
We assume $T,S$ are closed under logical consequences.

\begin{description}
\item[(\ref{lem:biembequivequnivth-3}) implies (\ref{lem:biembequivequnivth-1})]
By contraposition we prove 
$\neg$(\ref{lem:biembequivequnivth-1})$\to\neg$(\ref{lem:biembequivequnivth-3}).

Assume some universal sentence 
$\psi$ separates $T$ from $S$. Then for any model of $T$, all its substructures model
$\psi$, therefore they cannot be models of $S$.
\item[(\ref{lem:biembequivequnivth-1}) implies (\ref{lem:biembequivequnivth-3})]
By contraposition we prove 
$\neg$(\ref{lem:biembequivequnivth-3})$\to\neg$(\ref{lem:biembequivequnivth-1}).

Assume that for any model $\mathcal{M}$ of $S$ and $\mathcal{N}$ of $T$ 
$\mathcal{M}\not\sqsubseteq\mathcal{N}$. We must show that $T$ is $\Pi_1$-separated from $S$.

Given a $\tau$-structure $\mathcal{M}$ which models $S$,
let $\Delta_0(\mathcal{M})$ be the atomic diagram of $\mathcal{M}$ in the signature
$\tau\cup\mathcal{M}$.

The theory $T\cup\Delta_0(\mathcal{M})$ is inconsistent,
otherwise $\mathcal{M}$ embeds into some model of $T$:
let $\bar{\mathcal{Q}}$ be a $\tau\cup\mathcal{M}$-model 
of $\Delta_0(\mathcal{M})\cup T$ and
$\mathcal{Q}$ be the $\tau$-structure obtained from $\bar{\mathcal{Q}}$ omitting the interpretation of the constants not in $\tau$. Clearly $\mathcal{Q}$ models $T$.
The interpretation of the constants in $\tau\cup\mathcal{M}$ inside 
$\bar{\mathcal{Q}}$ defines a $\tau$-substructure of $\mathcal{Q}$ isomorphic to $\mathcal{M}$.

By compactness (since $\Delta_0(\mathcal{M})$ is closed under finite conjunctions)
 there is a quantifier free $\tau$-formula 
$\psi_{\mathcal{M}}(\vec{x})$ and $\vec{a}\in\mathcal{M}^{<\omega}$ such that
$T+\psi_{\mathcal{M}}(\vec{a})$ is inconsistent.
This gives that $T\vdash\neg\psi_{\mathcal{M}}(\vec{a})$.
Since $\vec{a}$ is a family of constants never occurring in $T$, we get that
$T\vdash\forall\vec{x}\neg\psi_{\mathcal{M}}(\vec{x})$ and 
$\mathcal{M}\models \exists \vec{x}\psi_{\mathcal{M}}(\vec{x})$.

The theory 
\[
S\cup\bp{\neg\exists \vec{x}\psi_{\mathcal{M}}(\vec{x}):\mathcal{M}\models S}
\]
is inconsistent,
since $\neg\exists \vec{x}\psi_{\mathcal{M}}(\vec{x})$ fails
 in any model $\mathcal{M}$ of $S$.

By compactness 
there is a finite set of formulae $\psi_{\mathcal{M}_1}\dots\psi_{\mathcal{M}_k}$ such that
\[
S+ \bigwedge\bp{\neg\exists \vec{x}_i\psi_{\mathcal{M}_i}(\vec{x}_i):i=1,\dots,k}
\]
is inconsistent. This gives that
\[
S\vdash \bigvee_{i=1}^k\exists \vec{x}_i\psi_{\mathcal{M}_i}(\vec{x}_i).
\]
The $\tau$-sentence $\psi:=\bigvee_{i=1}^k\exists \vec{x}_i\psi_{\mathcal{M}_i}(\vec{x}_i)$ holds 
in all models of $S$ and its negation
\[
\bigwedge\bp{\neg\exists \vec{x}_i\psi_{\mathcal{M}_i}(\vec{x}_i):i=1,\dots,k}
\]
is a conjunction of universal sentences derivable from $T$. Hence $\neg\psi$ separates $T$ from $S$.
\end{description}
\end{proof}

The following Lemma shows that models of $T_\forall$ can always be extended
to superstructures which model $T$.

\begin{lemma}\label{lem:keylemembed}
Let $T$ be a $\tau$-theory and $\mathcal{M}$ be a $\tau$-structure.
TFAE:
\begin{enumerate}
\item \label{lem:keylemembed-1}
$\mathcal{M}$ is a $\tau$-model of $T_\forall$.
\item \label{lem:keylemembed-2}
There exists $\mathcal{N}\sqsupseteq \mathcal{M}$
which models $T$.
\end{enumerate}
\end{lemma}
\begin{proof}
(\ref{lem:keylemembed-2}) implies (\ref{lem:keylemembed-1}) is trivial.

Conversely let $\Delta_0(\mathcal{M})$ be the $\tau\cup\mathcal{M}$-theory given by the 
atomic diagram of $\mathcal{M}$.
\begin{claim}
$T$ is not $\Pi_1$-separated from $\Delta_0(\mathcal{M})$ (in the signature $\tau\cup\mathcal{M}$).
\end{claim}
\begin{proof}
If not there are $\vec{a}\in \mathcal{M}^{<\omega}$,
and a quantifier free 
$\tau$-formula
$\phi(\vec{x},\vec{z})$ such that
\[
T\vdash \forall\vec{z}\phi(\vec{a},\vec{z}),
\]
while
\[
\Delta_0(\mathcal{M})\vdash\neg\forall\vec{z}\phi(\vec{a},\vec{z}). 
\]
The latter yields that
\[
\Delta_0(\mathcal{M})\vdash\exists\vec{x}\exists\vec{z}\neg\phi(\vec{x},\vec{z}),
\]
and therefore also that
\[
\mathcal{M}\models\exists\vec{x}\exists\vec{z}\neg\phi(\vec{x},\vec{z}).
\]

On the other hand, since the constants $\vec{a}$ do not appear in any of the sentences in $T$, we also get that
\[
T\vdash \forall\vec{x}\forall\vec{z}\phi(\vec{x},\vec{z}).
\]
This is a contradiction since $\mathcal{M}$ models $T_\forall$.
\end{proof}
By the Claim and Lemma \ref{lem:biembequivequnivth}
some $\tau\cup\mathcal{M}$-model $\bar{\mathcal{P}}$
of $\Delta_0(\mathcal{M})$ embeds into some $\tau\cup\mathcal{M}$-model 
$\bar{\mathcal{Q}}$ of 
$T$. Let $\mathcal{Q}$ be the $\tau$-structure obtained from $\bar{\mathcal{Q}}$ omitting the interpretation of the constants not in $\tau$. Then $\mathcal{Q}$ models $T$ and 
contains a substructure isomorphic to $\mathcal{M}$.
\end{proof}

\begin{corollary}[Resurrection Lemma]\label{cor.resurrlemma}
Assume $\mathcal{M}\prec_1\mathcal{N}$ are $\tau$-structures. 
Then there is $\mathcal{Q}\sqsupseteq\mathcal{N}$ which is an elementary extension of $\mathcal{M}$.
\end{corollary}
\begin{proof}
Let $T$ be the elementary diagram $\Delta_\omega(\mathcal{M})$  of 
$\mathcal{M}$ in the signature $\tau\cup\mathcal{M}$.
It is easy to check that any model of $T$ when restricted to the signature $\tau$ is an elementay extension
of $\mathcal{M}$.
Since $\mathcal{M}\prec_1\mathcal{N}$, the natural extension of $\mathcal{N}$ to a 
$\tau\cup\mathcal{M}$-structure
realizes the $\Pi_1$-fragment of $T$
in the signature $\tau\cup\mathcal{M}$.
Now apply the previous Lemma.
\end{proof}
The Resurrection Lemma motivates the resurrection axioms introduced by Hamkins and Johnstone in 
\cite{HAMJOH13}, and their iterated versions introduced by the author and Audrito in \cite{VIAAUD14}.

\subsection{Existentially closed structures}

The objective is now to isolate the ``generic'' models of some universal theory
$T$ (i.e. all axioms of $T$ are universal sentences). These are described by the 
$T$-existentially closed models.

\begin{definition}
Given a first order signature $\tau$,
let $T$ be any consistent $\tau$-theory.
A $\tau$-structure $\mathcal{M}$ is $T$-existentially closed ($T$-ec) if
\begin{enumerate}
\item
$\mathcal{M}$ can be embedded in a model of $T$.
\item
$\mathcal{M}\prec_{\Sigma_1}\mathcal{N}$ for all 
$\mathcal{N}\sqsupseteq\mathcal{M}$ which are models of $T$.
\end{enumerate}
\end{definition}

In general $T$-ec models need not be 
models\footnote{For example let 
$T$ be the theory of commutative rings with no zero divisors which are not fields
in the signature $(+,\cdot,0,1)$. Then 
the $T$-ec structures are exactly all the algebraically closed fields, and no $T$-ec model is a model of $T$. By Thm. \ref{Thm:mainthm-1}
$(H_{\omega_1},\sigma_{\omega}^V)$ is $S$-ec for $S$ the $\sigma_{\omega}$-theory of $V$, 
but it is not a model of $S$: the $\Pi_2$-sentence asserting that every set has countable transitive closure is true in $(H_{\omega_1},\sigma_{\omega}^V)$ but denied by $S$.} of $T$,
but only of their universal fragment. 
A standard diagonalization argument shows that for any theory $T$ there are $T$-ec
models, see Lemma \ref{lem:exTecmod} below or \cite[Lemma 3.2.11]{TENZIE}.

A trivial observation which will come handy in the sequel is the following:
\begin{fact}\label{fac:presECmod}
Assume $\mathcal{M}$ is a $T$-ec model and $S\supseteq T$ is such that
some $\mathcal{N}\sqsupseteq\mathcal{M}$ models $S$.
Then $\mathcal{M}$ is $S$-ec.
\end{fact}

\begin{proposition}\label{prop:pi2satKaiserhull}
Assume a $\tau$-structure $\mathcal{M}$ is $T$-ec. 
Then:
\begin{enumerate}
\item \label{rmk:itm1} $\mathcal{M}\models T_\forall$.
\item \label{rmk:itm0} $\mathcal{M}$ is also $T_\forall$-ec.
\item \label{rmk:itm2} If $\mathcal{N}\prec_{\Sigma_1}\mathcal{M}$, then $\mathcal{N}$ is also $T$-ec.
\item \label{rmk:itm3} Let $\forall\vec{x}\exists\vec{y}\psi(\vec{x},\vec{y},\vec{a})$ be a
$\Pi_2$-sentence with 
$\psi(\vec{x},\vec{y},\vec{z})$ quantifier free $\tau$-formula and
parameters $\vec{a}$ in $\mathcal{M}^{<\omega}$. Assume it holds in some  
$\mathcal{N}\sqsupseteq\mathcal{M}$ which models $T_\forall$, then it holds in
$\mathcal{M}$.
\item \label{rmk:itm4} 
Let $S$ be the $\tau$-theory of $\mathcal{M}$.
For any $\Pi_2$-sentence $\psi$ in the signature $\tau$ TFAE:
\begin{itemize}
\item $\psi$ holds in some model of $S_\forall$.
\item  $\psi$ holds in $\mathcal{M}$.
\end{itemize}
\end{enumerate}
\end{proposition}
\begin{proof}
\emph{}

\begin{description}

\item[(\ref{rmk:itm1})]
There is at least one super-structure of $\mathcal{M}$ which models
$T$, and any
$\psi\in T_\forall$ holds in this superstructure, hence in $\mathcal{M}$.

\item[(\ref{rmk:itm0})]
Assume $\mathcal{M}\sqsubseteq\mathcal{P}$ for some model $\mathcal{P}$ of $T_\forall$.
We must argue that $\mathcal{M}\prec_1\mathcal{P}$.

By Lemma~\ref{lem:keylemembed}, there is $\mathcal{Q}\sqsupseteq\mathcal{P}$ which models $T$.

Since $\mathcal{M}$ and $\mathcal{Q}$ are both models of $T$ and 
$\mathcal{M}$ is $T$-ec, we get the following diagram:
\[
		\begin{tikzpicture}[xscale=0.8,yscale=-0.4]
				\node (A0_0) at (0, 0) {$\mathcal{M}$};
				\node (A0_2) at (6, 0) {$\mathcal{Q}$};
				\node (A1_1) at (3, 3) {$\mathcal{P}$};
				\path (A0_0) edge [->]node [auto] {$\scriptstyle{\Sigma_1}$} (A0_2);
				\path (A1_1) edge [->]node [auto,swap] {$\scriptstyle{\sqsubseteq}$} (A0_2);
				\path (A0_0) edge [->]node [auto,swap] {$\scriptstyle{\sqsubseteq}$} (A1_1);
			\end{tikzpicture}
		\]
Then any 
$\Sigma_1$-formula $\psi(\vec{a})$ with $\vec{a}\in\mathcal{M}^{<\omega}$ realized in
$\mathcal{P}$ holds in $\mathcal{Q}$, and is therefore reflected to $\mathcal{M}$.
We are done by Tarski-Vaught's criterion.

\item[(\ref{rmk:itm2})]
Assume $\mathcal{N}\sqsubseteq\mathcal{P}$ for some model of $T_\forall$ $\mathcal{P}$.
Let $\Delta_0(\mathcal{P})$ be the atomic diagram
of $\mathcal{P}$ in the signature $\tau\cup\mathcal{P}\cup\mathcal{M}$ and 
$\Delta_0(\mathcal{M})$ be the atomic diagram
of $\mathcal{M}$ in the same signature\footnote{We are considering $\mathcal{P}\cup\mathcal{M}$ as the union of the domains of the structure 
$\mathcal{P},\mathcal{M}$ amalgamated over $\mathcal{N}$;
in particular we add a new constant for each element of  
$\mathcal{P}\setminus\mathcal{N}$, a new constant for each element of 
$\mathcal{M}\setminus\mathcal{N}$, a new constant for each element of $\mathcal{N}$.}.

\begin{claim}
$T_\forall\cup\Delta_0(\mathcal{P})\cup\Delta_0(\mathcal{M})$ is a consistent 
$\tau\cup\mathcal{M}\cup\mathcal{P}$-theory.
\end{claim}
\begin{proof}
Assume not. Find $\vec{a}\in (\mathcal{P}\setminus\mathcal{N})^{<\omega}$,  
$\vec{b}\in (\mathcal{M}\setminus\mathcal{N})^{<\omega}$,
$\vec{c}\in \mathcal{N}^{<\omega}$
 and $\tau$-formulae $\psi_0(\vec{x},\vec{z})$, 
$\psi_1(\vec{y},\vec{z})$ such that: 
\begin{itemize}
\item
$\psi_0(\vec{a},\vec{c})\in \Delta_0(\mathcal{P})$,
\item
$\psi_1(\vec{b},\vec{c})\in \Delta_0(\mathcal{M})$,
\item
$T\cup\bp{\psi_0(\vec{a},\vec{c}),\psi_1(\vec{b},\vec{c})}$
is inconsistent.
\end{itemize} 
Then
\[
T\vdash\neg\psi_0(\vec{a},\vec{c})\vee\neg\psi_1(\vec{b},\vec{c}).
\]
Since the constants appearing in $\vec{a},\vec{b},\vec{c}$ 
are never appearing in sentences of $T$,
we get that
\[
T\vdash\forall\vec{z}\,(\forall\vec{x}\neg\psi_0(\vec{x},\vec{z}))\vee
(\forall\vec{y}\neg\psi_1(\vec{y},\vec{z})).
\]
Since $\mathcal{P}$ models $T_\forall$, and 
\[
\mathcal{P}\models\psi_0(\vec{x},\vec{z})[\vec{x}/\vec{a},\vec{z}/\vec{c}],
\]
we get that
\[
\mathcal{P}\models\forall\vec{y}\neg\psi_1(\vec{y},\vec{c}).
\]
Therefore 
\[
\mathcal{N}\models\forall\vec{y}\neg\psi_1(\vec{y},\vec{c})
\]
being a substructure of $\mathcal{P}$, and so does $\mathcal{M}$
since $\mathcal{N}\prec_1\mathcal{M}$.
This contradicts $\psi_1(\vec{b},\vec{c})\in \Delta_0(\mathcal{M})$.
\end{proof}

If $\bar{\mathcal{Q}}$ is a model realizing 
$T_\forall\cup\Delta_0(\mathcal{P})\cup\Delta_0(\mathcal{M})$, and $\mathcal{Q}$ is the 
$\tau$-structure obtained forgetting the constant symbols not in $\tau$,
we get that:
\begin{itemize} 
\item
$\mathcal{P}$ and $\mathcal{M}$ are both substructures of $\mathcal{Q}$
containing $\mathcal{N}$ as a common substructure;
\item 
$\mathcal{N}\prec_1\mathcal{M}\prec_1\mathcal{Q}$, 
since $\mathcal{Q}$ realizes $T_\forall$ and $\mathcal{M}$ is $T_\forall$-ec.
\end{itemize}
We can now conclude that if a $\Sigma_1$-formula $\psi(\vec{c})$ for $\tau\cup\mathcal{N}$
with parameters in 
$\mathcal{N}$ holds in $\mathcal{P}$, it holds in $\mathcal{Q}$ as well (since
$\mathcal{Q}\sqsupseteq\mathcal{P}$), and therefore also in 
$\mathcal{N}$ (since $\mathcal{N}\prec_1\mathcal{Q}$).

\item[(\ref{rmk:itm3})] Observe that
for all $\vec{b}\in\mathcal{M}^{<\omega}$, $\exists \vec{y}\,\psi(\vec{b},\vec{y},\vec{a})$ holds in $\mathcal{N}$, and therefore in
$\mathcal{M}$, since $\mathcal{M}$ is $T$-ec; hence 
$\mathcal{M}\models \forall \vec{x}\exists \vec{y}\psi(\vec{x},\vec{y},\vec{a})$.

\item[(\ref{rmk:itm4})] 
First of all note that $\mathcal{M}$ is $S$-ec since $S\supseteq T$ 
(by  Fact \ref{fac:presECmod}).
By Lemma~\ref{lem:keylemembed} (applied to $S_\forall+\psi$ and $\mathcal{M}$)
any
$\Pi_2$-sentence $\psi$ for $\tau$ which holds in some model of $S_\forall$
holds in some model of $S_\forall$ which is a superstructure of
$\mathcal{M}$. Now apply \ref{rmk:itm3}.
\end{description}
\end{proof}

In particular a structure  is $T$-ec if and only if it is $T_\forall$-ec, and 
a $T$-ec structure realizes all $\Pi_2$-sentences which are consistent with its 
$\Pi_1$-theory.

We now show that any structure $\mathcal{M}$ 
can always be extended to a $T$-ec structure for 
any $T$ which is not separated from the $\Pi_1$-theory of $\mathcal{M}$.

\begin{lemma}\label{lem:exTecmod}\cite[Lemma 3.2.11]{TENZIE}
Given a first order $\tau$-theory $T$, 
any model of $T_\forall$
can be extended to a $\tau$-superstructure 
which is $T$-ec.
\end{lemma}
\begin{proof}
Given a model $\mathcal{M}$ of $T$, 
we construct an ascending chain of $T_\forall$-models as follows.
Enumerate all quantifier free $\tau$-formulae as
$\bp{\phi_\alpha(y,\vec{x}_\alpha):\alpha<|\tau|}$. 
Let $\mathcal{M}_0=\mathcal{M}$ have size $\kappa\geq|\tau|+\aleph_0$. 
Fix also some enumeration 
\begin{align*}
\pi:&\kappa\to|\tau|\times\kappa^2\\
&\alpha\mapsto(\pi_0(\alpha),\pi_1(\alpha),\pi_2(\alpha))
\end{align*} 
such that $\pi_2(\alpha)\leq\alpha$ for all $\alpha<\kappa$
and for each $\xi<|\tau|$, and $\eta,\beta<\kappa$
there are unboundedly many $\alpha<\kappa$ such that $\pi(\alpha)=(\xi,\eta,\beta)$.

Let now 
$\mathcal{M}_\eta$ with enumeration   
$\bp{\vec{m}^\xi_\eta:\xi<\kappa}$ of $\mathcal{M}_\eta^{<\omega}$ be given for all $\eta\leq\beta$.
If $\mathcal{M}_\beta$ is $T$-ec, stop the construction.
Else check whether $T_\forall\cup\Delta_0(\mathcal{M}_\beta)\cup\bp{\exists y\phi_{\pi_0(\alpha)}(y,\vec{m}^{\pi_1(\alpha)}_{\pi_2(\alpha)})}$
is a consistent $\tau\cup\mathcal{M}_\beta$-theory; if so let $\mathcal{M}_{\beta+1}$ have size $\kappa$ and
realize this theory.
At limit stages $\gamma$, let $\mathcal{M}_\gamma$ be the direct limit of the chain of $\tau$-structures 
$\bp{\mathcal{M}_\beta:\beta<\gamma}$. Then all $\mathcal{M}_\xi$ are models of $T_\forall$, and at some stage $\beta\leq\kappa$
$\mathcal{M}_\beta$ is $T_\forall$-ec (hence also $T$-ec), since all existential $\tau$-formulae with parameters in some 
$\mathcal{M}_\eta$
will be considered along the construction, and realized along the way if this is possible, 
and all $\mathcal{M}_\eta$ are always models of 
$T_\forall$ (at limit stages the ascending chain of $T_\forall$-models remains a $T_\forall$-model).
\end{proof}

Compare the above construction with the standard consistency proofs of bounded forcing axioms as given for example in
\cite[Section 2]{ASPBAGAPAL2001}.
In the latter case to preserve $T_\forall$ at limit stages we use iteration theorems\footnote{Assume $G$ is $V$-generic for a forcing
which is a 
limit of an iteration of length $\omega$ of forcings 
$\bp{P_n:n<\omega}$. In general
$H_{\omega_2}^{V[G]}$ is not given by the union of
 $H_{\omega_2}^{V[G\cap P_n]}$,
hence a subtler argument is needed to maintain that 
$H_{\omega_2}^{V[G]}$ preserves $T_\forall$.}.

\subsection{The Kaiser hull of a first order theory}

The Kaiser Hull of a theory $T$ describes the smallest elementary class containing all the
 ``generic'' structures for $T$. For most theories $T$ the models of the respective
 Kaiser hulls  realize 
exactly all  $\Pi_2$-sentences which are consistent with the
universal fragment of any extension of
$T$.

\begin{definition}\cite[Lemma 3.2.12, Lemma 3.2.13]{TENZIE}
Given a theory $T$ in a signature $\tau$, its Kaiser hull 
$\mathrm{KH}(T)$ is given by the
$\Pi_2$-sentences of $\tau$ which holds in all $T$-ec structures.
\end{definition}

\begin{definition}
A $\tau$-theory $T$ is $\Pi_n$-complete, if it is consistent and
for any $\Pi_n$-sentence 
either $\phi\in T$ or 
$\neg\phi\in T$.
\end{definition}

By Proposition \ref{prop:pi2satKaiserhull}.\ref{rmk:itm4} we get:

\begin{fact}\label{fac:charKaihull}
Given a $\Pi_1$-complete first order $\tau$-theory $T$, its Kaiser Hull is
a $\Pi_2$-complete $\tau$-theory defined by the request that
for any $\Pi_2$-sentence $\psi$
\[
\psi\in \mathrm{KH}(T)\quad \text{ if and only if }\quad \bp{\psi}\cup T_\forall\text{ is consistent}.
\] 
\end{fact}

In particular any model of the Kaiser hull of a $\Pi_1$-complete
$T$ realizes simultaneously 
all $\Pi_2$-sentences which are individually consistent with $T_\forall$.

For theories $T$ of interests to us their Kaiser hull can be described in the same terms, 
but the proof is much more delicate.
We start with the following weaker property which holds for arbitrary theories:
\begin{fact}\label{fac:charkaihullnonpi1comp}
Given a $\tau$-theory $T$, its Kaiser hull $\mathrm{KH}(T)$ contains
the set of $\Pi_2$-sentences $\psi$ for $\tau$ such that for all complete $S\supseteq T$,
$S_\forall\cup\bp{\psi}$ is consistent.
\end{fact}

\begin{proof}
Assume $\psi$ is a $\Pi_2$-sentence such that 
for all complete $S\supseteq T$,
$S_\forall\cup\bp{\psi}$ is consistent. We must show that
$\psi$ holds in all $T$-ec models.

Fix $\mathcal{M}$ an existentially closed model for $T$ (it exists by Lemma~\ref{lem:exTecmod}); 
we must show that
$\mathcal{M}\models \psi$.
Let $\mathcal{N}\sqsupseteq\mathcal{M}$ be a model of $T$ and $S$ be the 
$\tau$-theory of $\mathcal{N}$. Then $S$ is a complete theory and
$\mathcal{M}\models S_{\forall}$ since 
$\mathcal{M}\prec_1\mathcal{N}$ (being $T$-ec).
Since $S\supseteq T$, $\mathcal{M}$ is also $S$-ec (by Fact~\ref{fac:presECmod}).
Since  $S_\forall\cup\bp{\psi}$ is consistent, and $S_\forall$ is 
$\Pi_1$-complete, we obtain that
$\mathcal{M}$ models $\psi$, being an $S_\forall$-ec model, and using Fact \ref{fac:charKaihull}. 
\end{proof}

We will show in Lemma~\ref{fac:proofthm1-2}
that the set of $\Pi_2$-sentences described in the Fact provides an equivalent characterization of the Kaiser
hull for many theories admitting a model companion, among which those considered in the previous sections. 


\subsection{Model completeness}

It is possible (depending on the choice of the 
theory $T$)
that there are models of the Kaiser hull of $T$ which are not 
$T$-ec\footnote{This is the main issue we face in the proof of Thm. \ref{Thm:mainthm-4}: 
we cannot prove that the theory $T$ in its assumption has a model companion, 
we will only be able to compute that its Kaiser hull is described 
by \ref{Thm:mainthm-4}(\ref{Thm:mainthm-4-1}).}. 
Robinson has come up with two model theoretic properties 
(model completeness and model companionship)
which describe the case in which the models of the Kaiser hull of 
$T$ are exactly the class of 
$T$-ec models (even in case $T$ is not a complete theory).

\begin{definition}\label{def:modcompl}
A $\tau$-theory $T$ is \emph{model complete} if for all $\tau$-models $\mathcal{M}$ and $\mathcal{N}$ of $T$ we have that 
$\mathcal{M} \sqsubseteq \mathcal{N}$ implies $\mathcal{M} \prec \mathcal{N}$. 
\end{definition}

Remark that theories admitting quantifier elimination are automatically model complete.
On the other hand model complete theories need not be complete\footnote{For example the theory of 
algebraically closed fields is model complete, but algebraically closed fields of different characteristics are 
elementarily inequivalent.}.
However for theories $T$ which are $\Pi_1$-complete,
model completeness entails completeness:
any two models of a $\Pi_1$-complete, model complete $T$ share the same 
$\Pi_1$-theory, therefore 
if $T_1\supseteq T$ and $T_2\supseteq T$ with $\mathcal{M}_i$ a model of $T_i$, we can suppose (by Lemma 
\ref{lem:biembequivequnivth}) that $\mathcal{M}_1\sqsubseteq\mathcal{M}_2$. Since they are both models of $T$,
model completeness entails that $\mathcal{M}_1\prec \mathcal{M}_2$.

\begin{lemma}\cite[Lemma 3.2.7]{TENZIE}\label{lem:robtest}
(Robinson's test) Let $T$ be a $\tau$-theory. The following are equivalent:
\begin{enumerate}[(a)]
\item \label{lem:robtest-1} $T$ is model complete. 
\item \label{lem:robtest-2} Any model of $T$ is $T$-ec.
\item \label{lem:robtest-4} Each \emph{existential} $\tau$-formula $\phi(\vec{x})$ in free variables $\vec{x}$ 
is $T$-equivalent to a universal $\tau$-formula $\psi(\vec{x})$ in the same free variables. 
\item \label{lem:robtest-3} Each $\tau$-formula $\phi(\vec{x})$ in free variables $\vec{x}$ 
is $T$-equivalent to a universal $\tau$-formula $\psi(\vec{x})$ in the same free variables. 
\end{enumerate}
\end{lemma}
Remark that \ref{lem:robtest-3} (or \ref{lem:robtest-4}) shows that being a model complete $\tau$-theory $T$ is expressible by
a $\Delta_0(\tau,T)$-property in any model of $\ZFC$, hence it is absolute with respect 
to forcing.

\begin{proof}
\emph{}

\begin{description}
\item[\ref{lem:robtest-1} implies \ref{lem:robtest-2}]
Immediate.
\item[\ref{lem:robtest-2} implies \ref{lem:robtest-4}]
Fix an existential formula $\phi(\vec{x})$ in free variables $x_1,\dots,x_n$.
Let $\Gamma$ be the set of universal formulae $\theta(\vec{x})$ such that
\[
T\vdash\forall\vec{x}\,(\phi(\vec{x})\rightarrow \theta(\vec{x})).
\]
Note that $\Gamma$ is closed under finite conjunctions and disjunctions.
Let $\vec{c}=(c_1,\dots,c_n)$ be a finite set of new constant symbols
and $\Gamma(\vec{c})=\bp{\theta(\vec{c}):\, \theta(\vec{x})\in \Gamma}$.

It suffices to prove 
\begin{equation}\label{eqn:keyeq2-->3}
T\cup\Gamma(\vec{c})\models\phi(\vec{c});
\end{equation}
if this is the case, by compactness, a finite subset $\Gamma_0(\vec{c})$ of $\Gamma(\vec{c})$ is such that
\[
T\cup\Gamma_0(\vec{c})\models\phi(\vec{c});
\]
letting $\bar{\theta}(\vec{x}):=\bigwedge\bp{\psi(\vec{x}): \psi(\vec{c})\in\Gamma_0(\vec{c})}$,
the latter 
gives that 
\[
T\models \forall\vec{x}\,(\bar{\theta}(\vec{x})\rightarrow\phi(\vec{x}))
\]
(since the constants $\vec{c}$ do not appear in $T$). 

$\bar{\theta}(\vec{x})\in \Gamma$ is a universal formula witnessing \ref{lem:robtest-4} for $\phi(\vec{x})$.

So we prove (\ref{eqn:keyeq2-->3}):
\begin{proof}
Let $\mathcal{M}$ be a $\tau\cup\bp{c_1,\dots,c_n}$-model of 
$T\cup\Gamma(\vec{c})$. We must show that
$\mathcal{M}$ models $\phi(\vec{c})$.

The key step is to prove the following:

\begin{claim}
$T\cup\Delta_0(\mathcal{M})\cup\bp{\phi(\vec{c})}$
is consistent (where $\Delta_0(\mathcal{M})$ is the $\tau\cup\bp{c_1,\dots,c_n}$-atomic 
diagram of $\mathcal{M}$).
\end{claim}

Assume the Claim holds and let $\mathcal{N}$ realize the above theory.
Then 
\[
\mathcal{M}\sqsubseteq\mathcal{N}\restriction(\tau\cup\bp{\vec{c}}).
\]
Hence 
\[
\mathcal{M}\restriction\tau\sqsubseteq\mathcal{N}\restriction\tau.
\]
By \ref{lem:robtest-2} 
\[
\mathcal{M}\restriction\tau\prec_1\mathcal{N}\restriction\tau.
\]

Now let $b_1,\dots,b_n\in\mathcal{M}$ be the interpretations of $c_1,\dots,c_n$.
Then 
\[
\mathcal{N}\restriction\tau\models\phi(b_1,\dots,b_n).
\]
Since $\phi(\vec{x})$ is $\Sigma_1$ for $\tau$, we get that
\[
\mathcal{M}\restriction\tau\models\phi(b_1,\dots,b_n),
\]
hence
\[
\mathcal{M}\models\phi(c_1,\dots,c_n),
\]
and we are done.

So we are left with the proof of the Claim.
\begin{proof}
Let $\psi(\vec{x},\vec{y})$ be a quantifier free $\tau$-formula such that
$\psi(\vec{c},\vec{a})\in\Delta_0(\mathcal{M})$ for some $\vec{a}\in\mathcal{M}$.

Clearly $\mathcal{M}$ models
$\exists \vec{y}\psi(\vec{c},\vec{y})$.

Then the universal formula $\neg\exists \vec{y}\psi(\vec{c},\vec{y})\not\in\Gamma(\vec{c})$, since
$\mathcal{M}$ models its negation and $\Gamma(\vec{c})$ at the same time.

This gives that 
\[
T\not\vdash\forall\vec{x}\,(\phi(\vec{x})\rightarrow \neg\exists \vec{y}\psi(\vec{x},\vec{y})),
\]
i.e.
\[
T\cup \bp{\exists \vec{x}\,[\phi(\vec{x})\wedge\exists \vec{y}\psi(\vec{x},\vec{y})]}
\]
is consistent.

We conclude that
\[
T\cup\bp{\phi(\vec{c})\wedge\psi(\vec{c},\vec{a})}
\]
is consistent for any tuple $a_1,\dots,a_k\in\mathcal{M}$ and formula $\psi$ such that
$\mathcal{M}$ models $\psi(\vec{c},\vec{a})$
(since $\vec{c},\vec{a}$ are constants never appearing in the formulae of $T$).

This shows that $T\cup\Delta_0(\mathcal{M})\cup\bp{\phi(\vec{c})}$ is consistent.
\end{proof}

(\ref{eqn:keyeq2-->3}) is proved.
\end{proof}

\item[\ref{lem:robtest-4} implies \ref{lem:robtest-3}]
We prove by induction on $n$ that $\Pi_n$-formulae and $\Sigma_n$-formulae are $T$-equivalent to a $\Pi_1$-formula.

\ref{lem:robtest-4} gives the base case $n=1$ of the induction for $\Sigma_1$-formulae
and (trivially) for $\Pi_1$-formulae. 

Assuming we have proved the implication for all $\Sigma_{n}$ formulae for some fixed 
$n>0$, we obtain it  for $\Pi_{n+1}$-formulae $\forall\vec{x}\psi(\vec{x},\vec{y})$ (with $\psi(\vec{x},\vec{y})$
$\Sigma_n$)
applying the inductive assumptions to $\psi(\vec{x},\vec{y})$; 
next we 
observe that a $\Sigma_{n+1}$-formula is equivalent to the negation of a
$\Pi_{n+1}$-formula, which is in turn equivalent to the negation of a universal formula (by what we already argued),
which is equivalent to an existential formula, and thus equivalent to a universal formula (by \ref{lem:robtest-4}).

\item[\ref{lem:robtest-3} implies \ref{lem:robtest-1}]
By \ref{lem:robtest-3} every formula is $T$-equivalent both to a universal formula and to an existential formula (since its negation is $T$-equivalent to a universal formula).

This gives that $\mathcal{M}\prec \mathcal{N}$ whenever $\mathcal{M}\sqsubseteq\mathcal{N}$ are models of $T$, since truth of universal formulae is inherited by substructures, while truth of existential formulae pass to superstructures. 
\end{description}
\end{proof}

We will also need the following:

\begin{fact}\label{fac:proofthm1}
Let $\tau$ be a signature and 
$T$ a model complete $\tau$-theory. 
Let $\sigma\supseteq \tau$ be a signature and 
$T^*\supseteq T$ a $\sigma$-theory such that every
$\sigma$-formula is $T^*$-equivalent to a $\tau$-formula. Then $T^*$ is model complete.
\end{fact}
\begin{proof}
By the model completeness of $T$ and the assumptions on $T^*$ we get that every $\sigma$-formula is equivalent
to a $\Pi_1$-formula for $\tau\subseteq\sigma$. We conclude by Robinson's test.
\end{proof}

We will later show that model complete theories are the Kaiser hull of their universal fragment.
This will be part of a broad family of tameness
properties for first order theories which require a new concept in order to be properly formulated, that of model companionship.

\subsection{Model companionship}
Model completeness comes in pairs with another fundamental concept which generalizes to arbitrary first order 
theories the relation existing between algebraically closed fields and commutative rings without zero-divisors. As a matter of fact, the case described below
occurs when $T^*$ is the theory of algebraically closed fields
and 
$T$ is the theory of commutative rings with no zero divisors.

\begin{definition}\label{def:dmodcompship}
Given two theories $T$ and $T^*$ in the same language $\tau$, 
$T^*$ is the \emph{model companion} of $T$ if the following conditions holds:
\begin{enumerate}
\item Each model of $T$ can be extended to a model of $T^*$.
\item Each model of $T^*$ can be extended to a model of $T$.
\item $T^*$ is model complete. 
\end{enumerate}
\end{definition}

Different theories can have the same model companion, for example the theory of fields 
and the theory of commutative rings with 
no zero-divisors which are not fields both have the theory of algebraically closed fields  
as their model companion.

\begin{theorem}\cite[Thm 3.2.14]{TENZIE}
\label{thm:modcompletionchar}
Let $T$ be a first order theory. If its model companion $T^*$ exists, then
\begin{enumerate}
\item \label{thm:modcompletionchar-1} $T_{\forall} = T^*_{\forall}$.
\item \label{thm:modcompletionchar-2} $T^*$ is the theory of the existentially closed models of $T_{\forall}$. 
\end{enumerate}
\end{theorem}
\begin{proof}
\emph{}

\begin{enumerate}
\item By Lemma~\ref{lem:keylemembed}.
\item By  Robinson's test \ref{lem:robtest} $T^*$ is the theory realized exactly by the $T^*$-ec models; 
by Proposition \ref{prop:pi2satKaiserhull}(\ref{rmk:itm0}) $\mathcal{M}$ is $T^*$-ec if and only if it is
$T^*_\forall$-ec;
by (\ref{thm:modcompletionchar-1}) $T^*_\forall=T_\forall$.
\end{enumerate}
\end{proof}

An immediate by-product of the above Theorem is that
the model companion of a theory does not necessarily exist, but, if it does, it is unique
and is its Kaiser hull.

\begin{theorem} \cite[Thm. 3.2.9]{TENZIE}\label{thm:uniqmodcompan}
Assume $T$ has a model companion $T^*$. Then
$T^*$ is axiomatized by its $\Pi_2$-consequences and is the Kaiser hull of $T_\forall$.

Moreover $T^*$ is the unique model companion of $T$ and is characterized by the property of 
being the unique model complete theory $S$ such that $S_\forall=T_\forall$.
\end{theorem}
\begin{proof}
For quantifier free formulae $\psi(\vec{x},\vec{y})$ and $\phi(\vec{x},\vec{y})$
the assertion
\[
\forall\vec{x}\,[\exists\vec{y}\psi(\vec{x},\vec{y})\leftrightarrow\forall\vec{y}\phi(\vec{x},\vec{y})]
\]
is a $\Pi_2$-sentence.

Let $T^{**}$ be the theory given by the $\Pi_2$-consequences of $T^*$.

Since $T^*$ is model complete, by Robinson's test \ref{lem:robtest}\ref{lem:robtest-4},  
for any $\Sigma_1$-formula $\exists\vec{y}\psi(\vec{x},\vec{y})$
there is a universal formula $\forall\vec{y}\phi(\vec{x},\vec{y})$ such that
\[
\forall\vec{x}\,[\exists\vec{y}\psi(\vec{x},\vec{y})\leftrightarrow\forall\vec{y}\phi(\vec{x},\vec{y})]
\]
is in $T^{**}$.

Again by Robinson's test \ref{lem:robtest}\ref{lem:robtest-4} $T^{**}$ is model complete.

Now assume $S$ is a model complete theory such that 
$S_\forall=T_\forall$
Clearly $T^*_\forall=T^{**}_\forall=T_\forall$. 
By Robinson's test \ref{lem:robtest}\ref{lem:robtest-2} and Proposition \ref{prop:pi2satKaiserhull}(\ref{rmk:itm0}),
$S_\forall$ holds exactly in the $T_\forall$-ec models.
Hence $T^*=T^{**}$ since $T^*_\forall=T^{**}_\forall$.

This shows that any model complete theory is axiomatized by its $\Pi_2$-consequences, 
that the model companion $T^*$ of $T$ is unique, that $T^*$ is also the Kaiser hull of $T$
(being axiomatized by the $\Pi_2$-sentences which hold in all
$T$-ec-models), and is characterized by the propoerty of being the unique model complete
theory $S$ such that $T_\forall=S_\forall$.
\end{proof}

Thm. \ref{thm:uniqmodcompan} provides an
equivalent characterization of model companion theories
(which is expressible by a $\Delta_0$-property in parameters $T$ and $T^*$, hence absolute for transitive models of $\ZFC$). 

We use the following criteria for model companionship in the proofs of Theorems \ref{Thm:mainthm-3}, \ref{Thm:mainthm-5}, \ref{Thm:mainthm-1}.

\begin{lemma}\label{fac:proofthm1-2}
Let $T,T_0$ be $\tau$-theories 
with $T_0$ model complete.
Assume 
that for every complete $\tau$-theory $S\supseteq T$ there
is $\mathcal{M}$ which models $T_0+S_\forall$.
Then:
\begin{enumerate}
\item \label{fac:proofthm1-2-a}
$T^*=T_0+T_\forall$ is the model companion of $T$.
\item \label{fac:proofthm1-2-b}
$T^*$ is axiomatized by the the set of $\Pi_2$-sentences $\psi$ for $\tau$ such that 
$S_\forall\cup\bp{\psi}$ is consistent for all complete $S\supseteq T$.
\item \label{fac:proofthm1-2-c}
$T^*$ is axiomatized by the the set of $\Pi_2$-sentences $\psi$ for $\tau$ such that for all 
universal $\tau$-sentences $\theta$
$T_\forall+\theta+\psi$ is consistent if and only if so is $T_\forall+\theta$.
\end{enumerate}
\end{lemma}

\begin{proof}
By Fact \ref{fac:proofthm1} $T^*$ is model complete. 
\begin{enumerate}
\item
We need to show that any model of $T^*$ embeds into a model of $T$ and conversely.

Assume $\mathcal{N}$ models $T^*$.
Then $\mathcal{N}$ models $T_\forall$. By Lemma \ref{lem:keylemembed}
there exists $\mathcal{M}\sqsupseteq \mathcal{N}$
which models $T$.

Conversely let $\mathcal{M}$ model $T$ and
$S$ be the $\tau$-theory of $\mathcal{M}$.
By 
assumption there is $\mathcal{N}$ which models $T_0+S_\forall$
(but this $\mathcal{N}$ may not be a superstructure of $\mathcal{M}$).
Let $S^*$ be the $\tau$-theory of $\mathcal{N}$.
Then $S^*_\forall=S_\forall$, since $S_\forall$ and $S^*_\forall$ 
are $\Pi_1$-complete theories with $S^*_\forall\supseteq S_\forall$.
Moreover $S^*\supseteq T^*$, since $S_\forall\supseteq T_\forall$.

\begin{claim}
The $\tau\cup\mathcal{M}$-theory 
$S^*\cup\Delta_0(\mathcal{M})$
is consistent. 
\end{claim}

Assume the Claim holds, then $\mathcal{M}$ is a $\tau$-substructure of a model of 
$S^*\supseteq T^*$ and we are done.

\begin{proof}
If not there is $\psi(\vec{a})\in \Delta_0(\mathcal{M})$ such that 
$S^*\cup\bp{\psi(\vec{a})}$ is inconsistent.
This gives that
\[
S^*\vdash \neg \psi(\vec{a}).
\]
Since none of the constant in $\vec{a}$ occurs in $\tau$, we get that
\[
S^*\vdash \forall\vec{x}\neg \psi(\vec{x}),
\]
i.e. $\forall\vec{x}\neg \psi(\vec{x})\in S^*_\forall=S_\forall$.
But $\mathcal{M}$  models $S_\forall$ and
$\forall\vec{x}\neg \psi(\vec{x})$ fails in $\mathcal{M}$; a contradiction.
\end{proof}
\item
Assume $\psi\in T^*$ and $S$ is a complete extension of $T$, 
we must show that
$S_\forall+\psi$ is consistent.
By assumption there is $\mathcal{N}$ which models
$T^*+S_\forall$, and we are done.
\item Left to the reader.
\end{enumerate}
\end{proof}

\begin{remark}\label{rmk:keyrmkcharkaihull}
We do not know whether the characterization of the model companion of $T$ given in 
Lemma~\ref{fac:proofthm1-2}(\ref{fac:proofthm1-2-c})
can be proved for \emph{all} theories $T$ admitting a model companion: following the notation of 
the Lemma, it is conceivable that some $\tau$-theory $T$ has a model companion $T^*$ and there is 
some some univesal $\tau$-sentence $\theta$ such that for any 
model $\mathcal{M}$ of $T_\forall+\theta$ any superstructure of $\mathcal{M}$ which models
$T^*$ kills the truth of $\theta$.
In this case no $\Pi_2$-sentence in the Kaiser
hull of $T$ is consistent with the universal fragment of $T_\forall+\theta$.
\end{remark}

\subsection{Is model companionship a tameness notion?}\label{subsec:tameness-modcompan}

Model completeness and model companionship are ``tameness'' notion for first order theories
which must be handled with care.

\begin{proposition}\label{prop:quantelimallthe}
Given a signature $\tau$ consider the signature $\tau^*$ which adds an $n$-ary predicate
symbol $R_\phi$ for any $\tau$-formula $\phi(x_1,\dots,x_n)$ with displayed free variables.

Let $T_{\tau}$ be the following $\tau^*$-theory:
\begin{itemize}
\item
$\forall\vec x\,(\phi(\vec{x})\leftrightarrow R_\phi(\vec{x}))$ for all quantifier free $\tau$-formulae $\phi(\vec{x})$,
\item
$\forall\vec x\,[R_{\phi\wedge\psi}(\vec{x})\leftrightarrow (R_\psi(\vec{x})\wedge R_\phi(\vec{x}))]$
for all $\tau$-formulae $\phi(\vec{x}),\psi(\vec{x})$,
\item
$\forall\vec x\,[R_{\neg\phi}(\vec{x})\leftrightarrow \neg R_\phi(\vec{x})]$
for all $\tau$-formulae $\phi(\vec{x})$,
\item
$\forall\vec x\,[\exists yR_{\phi}(y,\vec{x})\leftrightarrow R_{\exists y\phi}(\vec{x})]$
for all $\tau$-formulae $\phi(y,\vec{x})$.
\end{itemize}

Then any $\tau$-structure $\mathcal{N}$ admits a unique extension to 
a $\tau^*$-structure $\mathcal{N}^*$ which models $T_\tau$.
Moreover every $\tau^*$-formula is $T_\tau$-equivalent to 
an atomic $\tau^*$-formula. In particular for any $\tau$-model $\mathcal{N}$, the algebras of
its $\tau$-definable subsets and of the $\tau^*$-definable subsets of $\mathcal{N}^*$ are the same.

Therefore for any consistent $\tau$-theory $T$, $T\cup T_{\tau}$ is consistent and 
admits quantifier elemination, hence
is model complete.
\end{proposition}

\begin{proof}
By an easy induction one can prove that any $\tau$-formula $\phi(\vec{x})$
is $T_\tau$-equivalent to the atomic $\tau^*$-formula $R_{\phi}(\vec{x})$.

Another simple inductive argument brings that any
$\tau^*$-formula $\phi(\vec{x})$ is $T_\tau$-equivalent to the 
$\tau$-formula obtained by replacing all symbols $R_\psi(\vec{x})$ occurring in
$\phi$ by the $\tau$-formula $\psi(\vec{x})$. Combining these observations together
we get that any $\tau^*$-formula is equivalent to an atomic 
$\tau^*$-formula. 

$T_{\tau}$ forces the $\mathcal{M}^*$-interpretation of any relation symbol $R_\phi(\vec{x})$
in $\tau^*\setminus\tau$ to
be the $\mathcal{M}$-interpretation of the $\tau$-formula $\phi(\vec{x})$ to which it is $T_{\tau}$-equivalent.
\end{proof}

Observe that the expansion of the language from $\tau$ to $\tau^*$ 
behaves well with respect to several model theoretic notions of tameness distinct 
from model completeness: for example $T$ is a \emph{stable} $\tau$-theory if and only if 
so is the $\tau^*$-theory $T\cup T_{\tau}$, the same holds for 
NIP-theories, or for $o$-minimal theories, or for $\kappa$-categorical theories.

The passage from $\tau$-structures to $\tau^*$-structures which model 
$T_{\tau}$ can have effects
on the embeddability relation; for example assume $\mathcal{M}\sqsubseteq\mathcal{N}$ is a non-elementary embedding of 
$\tau$-structures; then $\mathcal{M}^*\not\sqsubseteq\mathcal{N}^*$: if the non-atomic $\tau$-formula
$\phi(\vec{a})$ in parameter $\vec{a}\in \mathcal{M}^{<\omega}$ 
holds in $\mathcal{M}$ and does not hold in
$\mathcal{N}$, the atomic $\tau^*$-formula $R_\phi(\vec{a})$ holds in $\mathcal{M}^*$ and does not hold in
$\mathcal{N}^*$.

However if $T$ is a model complete $\tau$-theory, then for 
$\mathcal{M}\sqsubseteq\mathcal{N}$ $\tau$-models of $T$, we get that
$\mathcal{M}\prec\mathcal{N}$; this entails that $\mathcal{M}^*\sqsubseteq\mathcal{N}^*$, which (by the quantifier
elimination of $T\cup T_{\tau}$) gives that $\mathcal{M}^*\prec\mathcal{N}^*$. 
In particular for a model complete $\tau$-theory $T$ and $\mathcal{M},\mathcal{N}$ 
$\tau$-models of $T$,
$\mathcal{M}\sqsubseteq\mathcal{N}$ if and only if 
$\mathcal{M}^*\sqsubseteq\mathcal{N}^*$.

Let us now investigate the case of model companionship.
If $T$ is the model companion of $S$ with $S\neq T$ in the signature $\tau$, 
$T\cup T_\tau$ and $S\cup T_\tau$ are both model complete theories in the signature $\tau^*$. 
But $T\cup T_\tau$ cannot be the model companion of $S\cup T_\tau$, by uniqueness of the model companion,
since each of these theories is the model companion of itself and they are distinct.
Moreover if $T$ and $S$ are also complete, no $\tau^*$-model of $S\cup T_\tau$ can embed into a 
$\tau^*$-model of $T\cup T_\tau$:
since $T$ is the model companion of 
$S$ and $S\neq T$, $T_\forall=S_\forall$ and there is some $\Pi_2$-sentence $\psi$
$\forall x\exists y\phi(x,y)$ with $\phi$-quantifer free in $T\setminus S$.
Therefore $\forall x\,R_{\exists y\phi}(x) \in (T\cup T_\tau)_\forall \setminus (S\cup T_\tau)_\forall$; 
we conclude by Lemma \ref{lem:biembequivequnivth}, since $T\cup T_\tau$ and $S\cup T_\tau$ are complete, 
hence the above sentence separates $(T\cup T_\tau)_\forall$ from
$(S\cup T_\tau)_\forall$.

\subsection{Summing up}
The results of this section gives that for any $\tau$-theory $T$:
\begin{itemize}
\item The universal fragment of $T$ describes the family of substructures
of models of $T$, and the $T$-ec models realize all $\Pi_2$-sentences which are ``absolutely'' consistent with 
$T_\forall$ 
(i.e. consistent with the universal fragment of any extension of $T$).
\item Model companionship and model completeness describe (almost all) the cases 
in which the family of $\Pi_2$-sentences which are ``absolutely'' consistent with $T$ (as defined in the previous item)
describes the elementary class given by the $T$-ec structures.
\item One can always extend $\tau$ to a signature $\tau^*$ so that $T$ has a 
conservative extension to a $\tau^*$-theory $T^*$ which is model complete, but this process may be completely 
uninformative since it may completely destroy the
substructure relation existing between $\tau$-models of $T$ (unless $T$ is already model complete).
\item On the other hand 
for certain theories $T$ (as the axiomatizations of set theory 
considered in the present paper),
 one can unfold their ``tameness'' 
by carefully extending $\tau$ to a signature $\tau^*$ in which only certain $\tau$-formulae 
are made equivalent to atomic $\tau^*$-formulae. In the new signature $T$ can be extended to a conservative extension $T^*$ which has a model companion $\bar{T}$, while
this process has mild consequences on the $\tau^*$-substructure relation for models of $T^*_\forall$ 
(i.e. for the pairs of interest of $\tau$-models $\mathcal{M}_0\sqsubseteq\mathcal{M}_1$ of a suitable fragment of $T$, 
their unique extensions to $\tau^*$-models $\mathcal{M}^*_i$ are still models of $T^*_\forall$
and maintain that $\mathcal{M}^*_0\sqsubseteq\mathcal{M}^*_1$ also for $\tau^*$).
This gives useful structural information on the web of relations existing between $\tau^*$-models of $T^*_\forall$
(as outlined by Theorems \ref{Thm:mainthm-3}, \ref{Thm:mainthm-5},
\ref{Thm:mainthm-1}).
\item
Our conclusion is that model completeness and model companionship are tameness properties of elementary classes 
$\mathcal{E}$ defined by a theory $T$ rather than of the theory $T$ itself: 
these model-theoretic notions outline certain regularity patterns for the substructure relation on
models of $\mathcal{E}$, patterns which may be unfolded only when passing to a signature distinct 
from the one in which
$\mathcal{E}$ is first axiomatized (much the same way as it occurs for Birkhoff's 
characterization of algebraic varieties in terms of universal theories).

%
\item The results of the present paper shows that if we consider set theory together with large cardinal axioms as formalized in the signature $\sigma_\omega,\sigma_{\omega,\NS_{\omega_1}},\sigma_{\omega_1}$, we obtain (until now unexpected) tameness properties for this first order theory, properties 
which couple perfectly with well 
known (or at least published) generic absoluteness results.
We do not have an abstract model theoretic justification for selecting these signatures out of the 
continuum many
signatures which produce definable extensions of $\ZFC$. 
However the common practice of set theory (independently of our results) already 
motivate our choice, and our results validate it.
\end{itemize}
\section{Auxiliary results}\label{sec:auxres}

We collect here auxiliary results needed to prove
Theorems \ref{Thm:mainthm-3} and \ref{Thm:mainthm-1}.
We prove all these results working in ``standard''
models of $\ZFC$, i.e. we assume the models are well-founded. This is a practice we already adopted
in Section \ref{sec:Hkappa+}. We leave to the reader to remove this unnecessary assumption.

\subsection{Generalizations of Levy absoluteness}\label{subsec:genlevabs}

We start with a natural generalization of Levy's absoluteness we used in the proof of Thm. \ref{Thm:mainthm-1}.

\begin{lemma}\label{lem:levyabsHkappa+}
Let $\kappa$ be an infinite  cardinal
and $\mathcal{A}$ be any family of subsets of 
$\bigcup_{n\in\omega}\pow{\kappa}^n$.
Let $\tau_{\mathcal{A}}=\tau_{\ST}\cup\mathcal{A}$.

Then:
\[
(H_{\kappa^+}^V,\tau_{\mathcal{A}}^V)\prec_{\Sigma_1}
(V,\tau_{\mathcal{A}}^{V}).
\]
\end{lemma}

\begin{proof}
Assume for some $\tau_{\mathcal{A}}$-formula
$\phi(\vec{x},y)$ without quantifiers\footnote{A quantifier free $\tau_{A_1,\dots,A_k}$-formula is a boolean
combination of atomic $\tau_\ST$-formulae with formulae of type $A_j(\vec{x})$.
For example $\exists x\in y A(y)$ is not a quantifier free $\tau_{\ST}$-formula, and is actually equivalent to the 
$\Sigma_1$-formula
$\exists x(x\in y)\wedge A(y)$.}
and 
$\vec{a}\in H_{\kappa^+}$
\[
(V,\tau_{\mathcal{A}}^{V})\models\exists y\phi(\vec{a},y).
\]
Let $\alpha>\kappa$ be large enough  so that for some $b\in V_\alpha$
\[
(V,\tau_{\mathcal{A}}^{V})\models\phi(\vec{a},b).
\]
Then
\[
(V_\alpha,\tau_{\mathcal{A}}^{V})\models\phi(\vec{a},b).
\]
Let $A_1,\dots,A_k$ be the subsets of $\pow{\kappa}^{i_k}$ which are the predicates mentioned in $\phi$. 
By the downward Lowenheim-Skolem theorem, we  can find
$X\subseteq V_\alpha$ which is the domain of a $\tau_{A_1,\dots,A_k}$-elementary substructure of
\[
(V_\alpha,\tau_{\ST},A_1,\dots,A_k)
\]
such that $X$ is a set of size $\kappa$ containing $\kappa$ and such that
$A_1,\dots,A_k,\kappa,b,\vec{a}\in X$. 
Since $|X|=\kappa\subseteq X$, a standard argument shows that 
$H_{\kappa^+}\cap X$ is a transitive set, and that $\kappa^+$ is the least ordinal in
$X$ which is not contained in $X$.
Let $M$ be the transitive collapse of $X$ via the Mostowski collapsing map $\pi_X$.

We have that
the first ordinal moved by $\pi_X$ is $\kappa^+$ and
$\pi_X$ is the identity on $H_{\kappa^+}\cap X$. Therefore $\pi_X(a)=a$
for all 
$a \in H_{\kappa^+}\cap X$.
Moreover for $A\subseteq \pow{\kappa}^n$ in $X$
\begin{equation}\label{eqn:piXidonpowkappa}
\pi_X(A)=A\cap M.
\end{equation}
We prove equation (\ref{eqn:piXidonpowkappa}):
\begin{proof}
Since 
$X\cap V_{\kappa+1}\subseteq X\cap H_{\kappa^+}$,
$\pi_X$ is 
the identity on $X\cap H_{\kappa^+}$, and
$A\subseteq V_{\kappa+1}$,
we get that
\[
\pi_X(A)=\pi_X[A\cap X]=\pi_X[A\cap X\cap V_{\kappa+1}]=A\cap M\cap V_{\kappa+1}=A\cap M.
\]
\end{proof}
It suffices now to show that
\begin{equation}\label{eqn:keyeqlevabs}
(M,\tau_{\ST}^V,\pi_X(A_1),\dots,\pi_X(A_k))\sqsubseteq (H_{\kappa^+},\tau_{\ST}^V,A_1,\dots,A_k).
\end{equation}
Assume \ref{eqn:keyeqlevabs} holds; since $\pi_X$ is an isomorphism and $\pi_X(A_j)=\pi_X[A_j\cap X]$, we
get that 
\[
(M,\tau_{\ST}^V,\pi_X(A_1),\dots,\pi_X(A_k))\models\phi(\pi_X(b),\vec{a})
\]
since 
\[
(X,\tau_{\ST}^V,A_1\cap X,\dots,A_k\cap X)\models\phi(b,\vec{a}).
\]
By (\ref{eqn:keyeqlevabs}) we get that 
\[
(H_{\kappa^+},\tau_{\ST}^V,\pi_X(A_1),\dots,\pi_X(A_k))\models\phi(\pi_X(b),\vec{a})
\]
and we are done.

We prove (\ref{eqn:keyeqlevabs}):
since $M$ is transitive, any atomic $\tau_\ST$-formula (i.e. any $\Delta_0$-property)
holds true in $M$ if and if it holds in $H_{\kappa^+}$.
It remains to argue that the same occurs for the $\tau_{\mathcal{A}}$-formulae of type $A_j(x)$, i.e. that
$A_j\cap M=\pi_X(A_j)$ for all $j=1,\dots,n$; which is the case
by (\ref{eqn:piXidonpowkappa}).
\end{proof}

\begin{remark}
Key to the proof is the fact that subsets of $\kappa$
have bounded rank below $\kappa^+$.
If $A\subseteq H_{\kappa^+}$ has elements of 
unbounded rank, the equality $\pi_X(A)=A\cap M$ 
may fail: for example
if $A=H_{\kappa^+}$, $\pi_X(A)=H_{\kappa^+}\cap X$
while $A\cap M=M$.
This shows that \ref{eqn:keyeqlevabs} fails for this choice 
of $A$.
\end{remark}

\subsection{Universally Baire sets and generic absoluteness for second order number theory}\label{sec:genabssecordnumth}

We collect here the properties of universally Baire sets and the generic absoluteness results for second order 
number theory we need to prove Thm. \ref{Thm:mainthm-3}.

\begin{notation}
$\mathcal{A}\subseteq \bigcup_{n\in\omega}\pow{\kappa}^n$ is projectively closed
if it is closed under projections, finite unions, complementation, and permutations
(if $\sigma:n\to n$ is a permutation and $A\subseteq\pow{\kappa}^n$, 
$\hat{\sigma}[A]=\bp{(a_{\sigma(0)},\dots,a_{\sigma(n-1}):\, (a_0,\dots,a_{n-1})\in A}$).

Otherwise said, $\mathcal{A}$ 
is the class of lightface definable subsets of some signature on 
$\pow{\kappa}$.
\end{notation}

\subsection{Universally Baire sets}\label{subsec:univbaire}
Assuming large cardinals 
there is a very large sample of projectively closed families of subsets of $\pow{\omega}$ which are are ``simple'', 
hence it is natural to consider elements of these families
as atomic predicates. 

The exact definition of what is meant by a ``simple'' subset of $2^\omega$ 
is captured by the notion of universally Baire set.

Given a topological space $(X,\tau)$, $A\subseteq X$ is nowhere dense if its closure 
has a dense complement,
meager if it is the countable union of nowhere dense sets, with the Baire property if it 
has meager symmetric difference with
an open set.
Recall that $(X,\tau)$ is Polish if $\tau$ is a completely metrizable, separable topology on $X$.

\begin{definition}
(Feng, Magidor, Woodin) 
Given a Polish space $(X,\tau)$, $A\subseteq X$ is \emph{universally Baire} 
if for every compact Hausdorff space $(Y,\sigma)$ and
every continuous $f:Y\to X$ we have that $f^{-1}[A]$ has the Baire property in $Y$.

$\bool{UB}$ denotes the family of universally Baire subsets of $X$ for some Polish space $X$.
\end{definition}

We adopt the convention that $\mathsf{UB}$ denotes the class of universally Baire sets and of all elements of 
$\bigcup_{n\in\omega+1}(2^{\omega})^n$ (since the singleton of such elements are universally Baire sets).


The theorem below outlines three simple examples of projectively closed families of universally Baire sets
containing $2^\omega$.
\begin{theorem}\label{thm:UBsetsgenabs}
Let $T_0$ be the $\tau_\ST$-theory $\mathsf{ZFC^*}+$\emph{there are infinitely many 
Woodin cardinals and a measurable above}
and $T_1$ be the $\tau_\ST$-theory $\mathsf{ZFC^*}+$\emph{there are class many Woodin cardinals}.
\begin{enumerate}
\item \cite[Thm. 3.1.12, Thm. 3.1.19]{STATLARSON}
Assume $V$ models $T_0$. Then every projective subset of $2^\omega$ is universally Baire.
\item \cite[Thm. 3.3.3, Thm. 3.3.5, Thm. 3.3.6, Thm. 3.3.8, Thm. 3.3.13, Thm. 3.3.14]{STATLARSON}
Assume $V\models T_1$.
Then $\mathsf{UB}$ is projectively closed.
\end{enumerate}
\end{theorem}
%

To proceed further we now list the standard facts about universally Baire sets we will need:
 
 \begin{enumerate}
 \item\label{itm1:charUBsets} \cite[Thm. 32.22]{JECHST}
 $A\subseteq 2^{\omega}$ is universally Baire if and only if for each forcing notion $P$ there are 
 trees $T_A,S_A$ on $\omega\times\delta$ for some $\delta> |P|$
 such that $A=p[[T_A]]$ (where $p:(2\times\kappa)^\omega\to 2^{\omega}$ denotes the projection on the first component and $[T]$ denotes the body of the tree $T$), and
 \[
P\Vdash T_A\text{ and }S_A\text{ project to complements},
\]
by this meaning that for all $G$ $V$-generic for $P$
\[
V[G]\models (p[[T_A]]\cap p[[S_A]]=\emptyset)\wedge (p[[T_A]]\cup p[[S_A]]=(2^\omega)^{V[G]})
\]
\item 
Any two Polish spaces $X,Y$ of the same cardinality are Borel isomorphic \cite[Thm. 15.6]{kechris:descriptive}. 
\item 
Any Polish space is Borel isomorphic to a Borel subset of $[0;1]^\omega$ \cite[Thm. 4.14]{kechris:descriptive}, hence also to a Borel subset of $2^\omega$ (by the previous item).
 \item
Given $\phi:\mathbb{N}\to\mathbb{N}$, $\prod_{n\in\omega}2^{\phi(n)}$ is Polish
(it is actually homemomorphic to the union of $2^\omega$ with a countable Hausdorff space) 
\cite[Thm. 6.4, Thm. 7.4]{kechris:descriptive}.
\end{enumerate}

Hence it is not restrictive to focus just on universally Baire subsets of $2^\omega$ and of its
countable products, which is what we will do in the sequel.

\begin{notation}\label{not:notUBsetsVG}
Given $G$ a $V$-generic filter for some forcing $P\in V$, $A\in \UB^{V[G]}$ and
$H$ $V[G]$-generic filter for some forcing $Q\in V[G]$, 
\[
A^{V[G][H]}=\bp{r\in (2^\omega)^{V[G][H]}: V[G][H]\models r\in p[[T_A]]},
\]
where $(T_A,S_A)\in V[G]$ is any pair of trees as given in item \ref{itm1:charUBsets} above
such that $p[[T_A]]=A$ holds in $V[G]$,
and $(T_A,S_A)$ project to complements in $V[G][H]$.
\end{notation}

\subsection{Generic absoluteness for second order number theory}\label{subsec:genabssecnumth}



We decide to include a full proof of Woodin's generic absoluteness results
for second order number theory we used, it 
follows readily from \cite[Thm. 3.1.2]{STATLARSON} and the assumptions that there exists
class many Woodin limits of Woodin, we reduce these large cardinal assumptions to the existence of class many Woodin cardinals, while providing an alternative approach
to the proof of some of Woodin's result.
The theorem below is an improvement of  \cite[Thm. 3.1]{VIAMMREV}.

\begin{theorem}\label{thm:genabshomega1}
Assume in $V$ there are class many Woodin cardinals. Let $\mathcal{A}\in V$ be a 
family of universally Baire sets of $V$ and $\tau_{\mathcal{A}}=\tau_{\ST}\cup\mathcal{A}$. Let $G$ be $V$-generic for some forcing notion $P\in V$.

Then 
\[
(H_{\omega_1},\tau_{\mathcal{A}}^V)\prec(H_{\omega_1}^{V[G]},\tau_{\ST}^{V[G]},A^{V[G]}:A\in\mathcal{A}).
\]
\end{theorem}
\begin{proof}
%
%
We proceed by induction on $n$ to prove the following stronger assertion

\begin{claim}
Whenever $G$ is $V$-generic for some forcing notion $P$ in $V$ 
and $H$ is $V[G]$-generic for some forcing notion $Q$ in $V[G]$
\[
(H_{\omega_1}^{V[G]},\tau_{\ST}^{V[G]},A^{V[G]}: A\in\mathcal{A})\prec_n 
(H_{\omega_1}^{V[G][H]},\tau_{\ST}^{V[G][H]},A^{V[G][H]}: A\in\mathcal{A}).
\]
\end{claim}
\begin{proof}
It is not hard to 
check that for all  $A\in \mathcal{A}$, $A^{V[G]}=A^{V[G][H]}\cap V[G]$ 
(choose in $V$ a pair of trees $(T,S)$ such that $A=p[[T]]$ and the pair $(T,S)$ projects to complements in 
$V[G][H]$, and therefore also in $V[G]$).
Therefore  
$(H_{\omega_1}^{V[G]},\tau_{\ST}^{V[G]},A^{V[G]}: A\in\mathcal{A})$ is a
$\tau_{\mathcal{A}}$-substructure of  $(H_{\omega_1}^{V[G][H]},\tau_{\ST}^{V[G][H]},A^{V[G][H]}: A\in\mathcal{A})$.

This proves the base case of the induction.

We prove the successor step.

Assume that for any $G$ $V$-generic for some forcing $P\in V$ and $H$ $V[G]$-generic for some forcing $Q\in V[G]$
\[
(H_{\omega_1}^{V[G]},\tau_{\ST}^{V[G]},A^{V[G]}: A\in\mathcal{A})\prec_n 
(H_{\omega_1}^{V[G][H]},\tau_{\ST}^{V[G][H]},A^{V[G][H]}: A\in\mathcal{A}).
\]	
Fix $\bar{G}$ and $\bar{H}$ as in the assumptions of the Claim as witnessed by forcings $\bar{P}\in V$ and 
$\bar{Q}\in V[\bar{G}]$. 

We want to show that 
\[
		 (H_{\omega_1}^{V[\bar{G}]},\tau_{\ST}^{V[\bar{G}]},A^{V[\bar{G}]}: A\in\mathcal{A})\prec_{n+1} 
		 (H_{\omega_1}^{V[\bar{G}][\bar{H}]},\tau_{\ST}^{V[\bar{G}][\bar{H}]},A^{V[\bar{G}][\bar{H}]}: A\in\mathcal{A}).
\]
	
Let $\gamma$ be a Woodin cardinal of $V$ such that $\bar{P}\ast\dot{\bar{Q}}\in V_\gamma$ 
(where $\dot{\bar{Q}}\in V^P$ is chosen so that $\dot{\bar{Q}}_G=\bar{Q}$).

Then $\gamma$ is Woodin also in $V[\bar{G}]$. Let $K$ be $V[\bar{G}]$-generic for\footnote{$\tow{T}^{\omega_1}_\gamma$ denotes here the countable tower of height $\gamma$ denoted as $\mathbb{Q}_{<\gamma}$ in \cite[Section 2.7]{STATLARSON}.}  
$(\tow{T}^{\omega_1}_\gamma)^{V[\bar{G}]}$
with $\bar{H}\in V[K]$, so that $V[\bar{G}][K]=V[\bar{G}][\bar{H}][\bar{K}]$ for some 
$\bar{K}\in V[\bar{G}][K]$.	
	
		 Hence we have the following diagram:
		\[
			\begin{tikzpicture}[xscale=1.3,yscale=-0.6]
				\node (A0_0) at (0, 0) {$(H_{\omega_1}^{V[\bar{G}]},\tau_{\ST}^{V[\bar{G}]},A^{V[\bar{G}]}: A\in\mathcal{A})$};
				\node (A0_2) at (6, 0) {$(H_{\omega_1}^{V[\bar{G}][K]},\tau_{\ST}^{V[\bar{G}][K]},
				A^{V[\bar{G}][K]}: A\in\mathcal{A})$};
				\node (A1_1) at (3, 3) {$(H_{\omega_1}^{V[\bar{G}][\bar{H}]},
				\tau_{\ST}^{V[\bar{G}][\bar{H}]},A^{V[\bar{G}][\bar{H}]}: A\in\mathcal{A})$};
				\path (A0_0) edge [->]node [auto] {$\scriptstyle{\Sigma_\omega}$} (A0_2);
				\path (A1_1) edge [->]node [auto,swap] {$\scriptstyle{\Sigma_{n}}$} (A0_2);
				\path (A0_0) edge [->]node [auto,swap] {$\scriptstyle{\Sigma_{n}}$} (A1_1);
			\end{tikzpicture}
		\]
		obtained by inductive hypothesis applied both on $V[\bar{G}]$, $V[\bar{G}][\bar{H}]$ and on 
		$V[\bar{G}][\bar{H}]$, $V[\bar{G}][\bar{H}][\bar{K}]$, and using the fact that 
		$(H_{\omega_1}^{V[\bar{G}][K]},\tau_{\UB^{V[\bar{G}]}}^{V[\bar{G}][K]})$ 
		is a fully elementary 
		superstructure of $(H_{\omega_1}^{V[\bar{G}]},\tau_{\UB^{V[\bar{G}]}}^{V[\bar{G}]})$ \cite[Thm. 2.7.7, Thm. 2.7.8]{STATLARSON}.

		Let $\phi \equiv \exists x \psi(x)$ be any $\Sigma_{n+1}$ formula for $\tau_{\mathcal{A}}$
		with parameters in $H_{\omega_1}^{V[\bar{G}]}$.
		First suppose that $\phi$ holds in $(H_{\omega_1}^{V[\bar{G}]},\tau_{\ST}^{V[\bar{G}]},A^{V[\bar{G}]}: A\in\mathcal{A})$, 
		and fix $\bar{a} \in V[\bar{G}]$ such that $\psi(\bar{a})$ holds
		in $(H_{\omega_1}^{V[\bar{G}]},\tau_{\ST}^{V[\bar{G}]},A^{V[\bar{G}]}: A\in\mathcal{A})$. 
		Since 
		\[
		(H_{\omega_1}^{V[\bar{G}]},\tau_{\ST}^{V[\bar{G}]},A^{V[\bar{G}]}: A\in\mathcal{A})\prec_n
		(H_{\omega_1}^{V[\bar{G}][\bar{H}]},\tau_{\ST}^{V[\bar{G}][\bar{H}]},A^{V[\bar{G}][\bar{H}]}: A\in\mathcal{A}),
		\]
		we conclude that $\psi(\bar{a})$ holds
		in $(H_{\omega_1}^{V[\bar{G}][\bar{H}]},\tau_{\ST}^{V[\bar{G}][\bar{H}]},A^{V[\bar{G}][\bar{H}]}: A\in\mathcal{A})$, 
		hence so does
		$\phi$.
		
		Now suppose that $\phi$ holds in 
		$(H_{\omega_1}^{V[\bar{G}][\bar{H}]},\tau_{\ST}^{V[\bar{G}][\bar{H}]},A^{V[\bar{G}][\bar{H}]}: A\in\mathcal{A})$
		 as witnessed by $\bar{a} \in H_{\omega_1}^{V[\bar{G}][\bar{H}]}$. 
		 
		 Since 
		 \[
		 (H_{\omega_1}^{V[\bar{G}][\bar{H}]},\tau_{\ST}^{V[\bar{G}][\bar{H}]},A^{V[\bar{G}][\bar{H}]}: A\in\mathcal{A})
		 \prec_n
		 (H_{\omega_1}^{V[\bar{G}][K]},\tau_{\ST}^{V[\bar{G}][K]},A^{V[\bar{G}][K]}: A\in\mathcal{A}),
		 \]
		 it follows that $\psi(\bar{a})$ holds in 
		 $(H_{\omega_1}^{V[\bar{G}][K]},\tau_{\ST}^{V[\bar{G}][K]},A^{V[\bar{G}][K]}: A\in\mathcal{A})$, 
		 hence so does $\phi$. 
		 Since 
		 \[
		 (H_{\omega_1}^{V[\bar{G}]},\tau_{\ST}^{V[\bar{G}]},A^{V[\bar{G}]}: A\in\mathcal{A})\prec
		 (H_{\omega_1}^{V[\bar{G}][K]},\tau_{\ST}^{V[\bar{G}][K]},A^{V[\bar{G}][K]}: A\in\mathcal{A}),
		 \]
		the formula $\phi$ holds also in $(H_{\omega_1}^{V[\bar{G}]},\tau_{\ST}^{V[\bar{G}]},A^{V[\bar{G}]}: A\in\mathcal{A})$.
		
		Since $\phi$ is arbitrary, this shows that 
		\[
		 (H_{\omega_1}^{V[\bar{G}][\bar{H}]},\tau_{\ST}^{V[\bar{G}]},A^{V[\bar{G}]}: A\in\mathcal{A})\prec_{n+1}
		 (H_{\omega_1}^{V[\bar{G}][\bar{H}]},\tau_{\ST}^{V[\bar{G}][\bar{H}]}, A^{V[\bar{G}][\bar{H}]}: A\in\mathcal{A}),
		 \]
		concluding the proof of the inductive step for $\bar{G}$ and $\bar{H}$.
		
		Since we have class many Woodin,
		this argument is modular in $\bar{G},\bar{H}$ as in the assumptions of the inductive step,
		because we can always find some Woodin cardinal $\gamma$
		of $V$ which remains Woodin in $V[\bar{G}]$ and is of size larger than the poset 
		in $V[\bar{G}]$ for which 
		$\bar{H}$ is $V[\bar{G}]$-generic.
		The proof of the inductive step is completed.
		\end{proof}
%

\end{proof}


\section{Further results} \label{sec:furtherresults}

We introduce without a few comments the results whose proof is defered to a second paper, together with the relevant terminology and definitions. 
The following supplements
Notation \ref{not:keynotation}.

\begin{Notation}\label{not:keynotation-2}
\emph{}

\begin{itemize}
\item
$\tau_{\NS_{\omega_1}}$ is the signature $\tau_\ST\cup\bp{\omega_1}\cup\bp{\NS_{\omega_1}}$ with $\omega_1$ a constant symbol, $\NS_{\omega_1}$ a unary predicate symbol.
\item
$T_{\NS_{\omega_1}}$ is the $\tau_{\NS_{\omega_1}}$-theory
given by $T_\ST$ together with the axioms
\[
\omega_1\text{ is the first uncountable cardinal},
\]
\[
\forall x\;[(x\subseteq\omega_1\text{ is non-stationary})\leftrightarrow\NS_{\omega_1}(x)].
\]

\item
$\ZFC^-_{\NS_{\omega_1}}$ is the $\tau_{\NS_{\omega_1}}$-theory 
\[
\ZFC^-_\ST+T_{\NS_{\omega_1}}.
\]
\item
Accordingly we define 
$\ZFC_{\NS_{\omega_1}}$.
\end{itemize}
\end{Notation}

\begin{Theorem}\label{Thm:mainthm-7}
Let $\mathcal{V}=(V,\in)$ be a model of 
\[
\ZFC+\maxUB+\emph{there is a supercompact cardinal and class many Woodin cardinals},
\]
and $\UB$ denote the family of universally Baire sets in 
$V$.

TFAE
\begin{enumerate}
\item\label{thm:char(*)-modcomp-1}
$(V,\in)$ models $\stUB$;
\item\label{thm:char(*)-modcomp-2}
$\NS_{\omega_1}$ is precipitous\footnote{See \cite[Section 1.6, pag. 41]{STATLARSON}  for a definition of precipitousness and a discussion of its properties. A key observation is that $\NS_{\omega_1}$ being precipitous is independent of $\mathsf{CH}$ (see for example \cite[Thm. 1.6.24]{STATLARSON}), while $\stUB$ entails $2^{\aleph_0}=\aleph_2$ (for example by the results of \cite[Section 6]{HSTLARSON}).

Another key point is that we stick to the formulation of $\Pmax$ as in \cite{HSTLARSON} so to 
be able in its proof to quote verbatim from \cite{HSTLARSON} all the relevant results on $\Pmax$-preconditions we will use.
It is however possible to  develop $\Pmax$ focusing on Woodin's countable tower rather than 
on the precipitousness of $\NS_{\omega_1}$ to develop the notion of $\Pmax$-precondition. Following this approach in
all its scopes, one should be able to reformulate Thm. \ref{Thm:mainthm-7}(\ref{thm:char(*)-modcomp-2}) 
omitting the request that
$\NS_{\omega_1}$ is precipitous. We do not explore this venue any further neither here nor in the sequel of this paper.} and
the $\tau_{\NS_{\omega_1}}\cup\UB$-theory of $V$ has as model companion the
$\tau_{\NS_{\omega_1}}\cup\UB$-theory of $H_{\omega_2}$.
\end{enumerate}
\end{Theorem}

Here is the definition of $\maxUB$ and $\stUB$:

\begin{Definition}\label{Keyprop:maxUB}
$\maxUB$: 
There are class many Woodin cardinals in $V$, and  for all 
$G$ $V$-generic for some forcing notion $P\in V$:
\begin{enumerate}
\item\label{Keyprop:maxUB-1}
Any subset of $(2^\omega)^{V[G]}$ definable in $(H_{\omega_1}^{V[G]}\cup\mathsf{UB}^{V[G]},\in)$ 
is universally Baire in $V[G]$.
\item\label{Keyprop:maxUB-2}
Let $H$ be $V[G]$-generic for some forcing notion $Q\in V[G]$. 
Then\footnote{Elementarity is witnessed via the map defined by $A\mapsto A^{V[G][H]}$ for
$A\in \UB^{V[G]}$ and the identity on
$H_{\omega_1}^{V[G]}$ (See Notation \ref{not:notUBsetsVG}
for the definition of $A^{V[G][H]}$).}:
\[
(H_{\omega_1}^{V[G]}\cup\UB^{V[G]},\in) \prec 
(H_{\omega_1}^{V[G][H]}\cup\UB^{V[G][H]},\in).
\] 
\end{enumerate}
\end{Definition}
$\maxUB$ is a form of sharp for the universally Baire sets
(a slight weakening of the conclusion of \cite[Thm. 4.17]{STATLARSON}).
It holds in any forcing extension of $V$ where a supercompact of $V$ becomes countable.
We will comment in details on $\maxUB$ in the sequel of this paper.

See \cite{HSTLARSON} for a definition of $\Pmax$ and \cite[Section 1.6, pag. 39]{STATLARSON} 
for a definition of saturated ideal on $\omega_1$.
\begin{Definition}
Let $\mathcal{A}$ be a family of dense subsets of $\Pmax$.
\begin{itemize}
\item
$\stA$ holds if $\NS_{\omega_1}$ is saturated and 
there exists a filter $G$ on $\Pmax$ meeting all the dense sets in 
$\mathcal{A}$.
\item
$\stUB$ holds
if $\NS_{\omega_1}$ is saturated and there exists an $L(\UB)$-generic filter $G$ on $\Pmax$. 
\end{itemize}
\end{Definition}

Woodin's definition of $(*)$ \cite[Def. 7.5]{HSTLARSON}
is equivalent to $\stA+$\emph{there are class many Woodin cardinals} 
for $\mathcal{A}$ the family of 
dense subsets of $\Pmax$ existing in $L(\mathbb{R})$.


\begin{Theorem}\label{Thm:mainthm-8}
Assume $V$ models that there are class many Woodin cardinals and $\UB$ is the family of universally Baire sets in $V$.
Then the $\Pi_1$-theory of $V$ for the language $\tau_{\NS_{\omega_1}}\cup\UB$ is invariant under set sized forcings.
\end{Theorem}

\begin{Notation}
\emph{}

\begin{itemize}
\item Given a family $\mathcal{A}$ of predicate symbols:
\begin{itemize}
\item
$\sigma_{\mathcal{A}}=\tau_{\ST}\cup\mathcal{A}$,
\item
$\sigma_{\mathcal{A},\NS_{\omega_1}}=\tau_{\NS_{\omega_1}}\cup\mathcal{A}$,
\item
$\sigma_\omega$ is $\sigma_{\mathcal{A}}$ for 
$\mathcal{A}=\sigma_\ST$,
\item
$\sigma_{\omega,\NS_{\omega_1}}$ is
$\sigma_{\mathcal{A},\NS_{\omega_1}}$ for 
$\mathcal{A}=\sigma_\ST$.
\end{itemize}

\item
Let $\UB$ denote the family of universally Baire sets, and $L(\UB)$ denote the smallest transitive model of $\ZF$ which contains $\UB$.
\smallskip

$T_\lUB$ is the 
$\sigma_{\omega,\NS_{\omega_1}}$-theory 
given by the axioms
\[
\forall x_1\dots x_n\,[S_\psi(x_1,\dots,x_n)\leftrightarrow 
(\bigwedge_{i=1}^n x_i\subseteq \omega^{<\omega}\wedge \psi^{L(\UB)}(x_1,\dots,x_n))]
\]
as $\psi$ ranges over the $\in$-formulae.

\item
$\ZFC^{*-}_{\lUB}$ is the $\sigma_{\omega}$-theory 
\[
\ZFC^-_{\ST}\cup T_\lUB;
\]

\item
$\ZFC^{*-}_{\lUB,\NS_{\omega_1}}$ is the $\sigma_{\omega,\NS_{\omega_1}}$-theory 
\[
\ZFC^-_{\NS_{\omega_1}}\cup T_\lUB;
\]
\item
Accordingly we define 
$\ZFC_{\NS_{\omega_1}}$, $\ZFC^{*}_{\lUB}$, $\ZFC^*_{\lUB,\NS_{\omega_1}}$.
\end{itemize}
\end{Notation}

\begin{Theorem} \label{Thm:mainthm-6}
Let $T$ be any $\sigma_{\omega,\NS_{\omega_1}}$-theory extending 
\[
\ZFC^*_{\lUB,\NS_{\omega_1}}+\maxUB+\text{ there is a supercompact cardinal and  class many Woodin cardinals}.
\]
Then $T$ has a model companion $T^*$. 

Moreover TFAE for any for any $\Pi_2$-sentence $\psi$ for 
$\sigma_{\omega,\NS_{\omega_1}}$:
\begin{enumerate}[(A)]
\item\label{Thm:mainthm-1A} 
$T^*\vdash \psi$.
\item\label{Thm:mainthm-1B}
 For any complete theory 
\[
S\supseteq T,
\] 
$S_\forall\cup\bp{\psi}$ is consistent.
\item \label{Thm:mainthm-1C}
$T$ proves
\[
\exists P \,(P\text{ is a partial order }\wedge 
\Vdash_P\psi^{\dot{H}_{\omega_2}}).
\]
\item \label{Thm:mainthm-1D}
$T_\forall+\ZFC^*_{\lUB,\NS_{\omega_1}}+\maxUB+\stUB\vdash \psi^{H_{\omega_2}}$.
\item \label{Thm:mainthm-1E}
$T$ proves that 
\[
(\Pmax\Vdash\psi^{\dot{H}_{\omega_2}})^{L(\UB)}.
\]
\end{enumerate}

\end{Theorem}

We immediately obtain Thm. \ref{Thm:mainthm-4}  as a corollary of 
Thm. \ref{Thm:mainthm-6} and Thm. \ref{Thm:mainthm-8}:
\begin{proof}
Note that every lightface projective set is in 
$L(\UB)$ (since the quantifer defining the set range over 
$\pow{\omega}\subseteq L(\UB)$;
hence we can assume that $\ZFC^*_{\omega,\NS_{\omega_1}}$ is a fragment of 
$\ZFC^*_{\lUB,\NS_{\omega_1}}$: the interpretation of $S_\theta$ according to 
$\ZFC^*_{\omega,\NS_{\omega_1}}$
is the same of $S_{\theta^{\pow{\omega^{<\omega}}}}$ according to 
$\ZFC^*_{\omega,\NS_{\omega_1}}$
which has the same interpretation of $S_{\theta^{\pow{\omega^{<\omega}}}}$ according to 
$\ZFC^*_{\lUB,\NS_{\omega_1}}$.
Therefore a
$\Pi_2$-sentence for $\sigma_\omega$ in the theory $\ZFC^*_{\omega,\NS_{\omega_1}}$ 
can be regarded as
a $\Pi_2$-sentence also for the theory $\ZFC^*_{\lUB,\NS_{\omega_1}}$.
\begin{description}
\item[(\ref{Thm:mainthm-4-1}) implies (\ref{Thm:mainthm-4-3})]
If $P$ forces $\MM^{++}$, by Asper\`o and Schindler's result, $P\Vdash\stUB$; hence 
$P\Vdash\psi^{H_{\omega_2}}$ by (\ref{Thm:mainthm-4-1}).

\item[(\ref{Thm:mainthm-4-3}) implies (\ref{Thm:mainthm-4-2})]
Given some complete $S\supseteq T$, and a model $\mathcal{M}$ of $S$, 
find $\mathcal{N}$ forcing extension of $\mathcal{M}$ which models $\psi^{H_{\omega_2}}$.
By Thm. \ref{Thm:mainthm-8} and Lemma \ref{lem:levyabsHkappa+}, 
$H_{\omega_2}^{\mathcal{N}}\models S_\forall$ and we are done.

\item[(\ref{Thm:mainthm-4-2}) implies (\ref{Thm:mainthm-4-1})]
assume 
$\mathcal{M}$ models
\[
T_\forall+\ZFC^*_{\omega,\NS_{\omega_1}}+\theta_{\bool{SC}}+\stUB;
\]
find $\mathcal{N}$ forcing extension of $\mathcal{M}$ which models
\[
T_\forall+\ZFC^*_{\omega,\NS_{\omega_1}}+\maxUB.
\]
By Thm. \ref{Thm:mainthm-8} and (\ref{Thm:mainthm-4-2}), 
$\psi$ is consistent with the $\Pi_1$-theory of $\mathcal{N}$. 
By the equivalence
of $\ref{Thm:mainthm-1A}$ with $\ref{Thm:mainthm-1D}$ of Thm. \ref{Thm:mainthm-1} applied
to the $\Pi_1$-complete theory of $\mathcal{N}$,
we get that $\mathcal{N}$ models $\psi^{H_{\omega_2}}$ is forcible by $\Pmax$ over $L(\UB)$.
Since all the universally Baire predicates predicates appearing in $\psi$ are projective and lightface definable,
$\mathcal{N}$ models $\psi^{H_{\omega_2}}$ is forcible by $\Pmax$ over $L(\mathbb{R})$.
Since $L(\mathbb{R})^{\mathcal{M}}$ and $L(\mathbb{R})^{\mathcal{N}}$ are elementarily equivalent 
(without any need to appeal to $\maxUB$, but just to $\theta_{\bool{SC}}$ and \cite[Thm. 3.1.2]{STATLARSON}), we get 
that $\mathcal{M}$ models $\psi^{H_{\omega_2}}$ is forcible by $\Pmax$ over $L(\mathbb{R})$.
Since $\mathcal{M}\models\stUB$, we conclude that $\psi^{H_{\omega_2}}$ holds in 
$\mathcal{M}$.
\end{description}
\end{proof}

\bibliographystyle{plain}
	\bibliography{Biblio}

\end{document}